\documentclass{amsart}%
\usepackage{amsmath}
\usepackage{color}
\usepackage{amssymb}
\usepackage{amsfonts}
\usepackage{graphicx}%
\newtheorem{theorem}{Theorem}[section]
\theoremstyle{plain}
\newtheorem{acknowledgement}{Acknowledgement}

\newtheorem{corollary}[theorem]{Corollary}

\newtheorem{definition}[theorem]{Definition}

\newtheorem{lemma}[theorem]{Lemma}

\newtheorem{proposition}[theorem]{Proposition}

\numberwithin{equation}{section}

\def \v{\mathcal V}

\begin{document}
\title[Nodal solutions to Yamabe type equations]{Multiplicity of nodal solutions to the Yamabe problem}
\author{M\'{o}nica Clapp}
\address{Instituto de Matem\'{a}ticas, Universidad Nacional Aut\'{o}noma de M\'{e}xico,
Circuito Exterior, C.U., 04510 M\'{e}xico D.F., Mexico}
\email{monica.clapp@im.unam.mx}
\author{Juan Carlos Fern\'{a}ndez}
\address{Instituto de Matem\'{a}ticas, Universidad Nacional Aut\'{o}noma de M\'{e}xico,
Circuito Exterior, C.U., 04510 M\'{e}xico D.F., Mexico}
\email{jcfmor@im.unam.mx}
\thanks{This research was partially supported by CONACYT grant 237661 (Mexico) and
UNAM-DGAPA-PAPIIT grant IN104315 (Mexico). J.C. Fern\'{a}ndez was partially
supported by a PhD fellowship from UNAM-DGAPA}
\date{\today}
\maketitle

\begin{abstract}
\smallskip Given a compact Riemannian manifold $(M,g)$ without boundary of
dimension $m\geq3$ and under some symmetry assumptions, we establish existence
of one positive and multiple nodal solutions to theYamabe-type equation
\[
-\text{div}_{g}(a\nabla u)+bu=c|u|^{2^{\ast}-2}u\quad\text{ on }M,
\]
where $a,b,c\in\mathcal{C}^{\infty}(M),$ $a$ and $c$ are positive, $-$%
div$_{g}(a\nabla)+b$ is coercive, and $2^{\ast}=\frac{2m}{m-2}$ is the
critical Sobolev exponent.

In particular, if $R_{g}$ denotes the scalar curvature of $(M,g)$, we give
conditions which guarantee that the Yamabe problem
\[
\Delta_{g}u+\frac{m-2}{4(m-1)}R_{g}u=\kappa u^{2^{\ast}-2}\quad\text{ on }M
\]
admits a prescribed number of nodal solutions.



\bigskip

\end{abstract}

\baselineskip15pt

\section{Introduction and statement of results}

Given a compact Riemannian manifold $(M,g)$ without boundary of dimension
$m\geq3$, the Yamabe problem consists in finding a metric $\hat{g}$
conformally equivalent to $g$ with constant scalar curvature. If $\hat{g}$ is
conformally equivalent to $g$ we can write it as $\hat{g}=u^{4/(m-2)}g$ with
$u\in\mathcal{C}^{\infty}(M),$ $u>0.$ Then, $\hat{g}$ has constant scalar
curvature $\mathfrak{c}_{m}\kappa$\ \ iff $u$ is a positive solution to the
problem
\begin{equation}
\Delta_{g}u+\mathfrak{c}_{m}R_{g}u=\kappa\left\vert u\right\vert ^{2^{\ast}%
-2}u,\qquad u\in\mathcal{C}^{\infty}(M)\text{,} \tag{$\mathcal{Y}%
_g$}\label{yamabe}%
\end{equation}
where $\Delta_{g}=-\,$div$_{g}\nabla_{g}$ is the Laplace-Beltrami operator,
$\mathfrak{c}_{m}:=\frac{m-2}{4(m-1)}$, $R_{g}$ is the scalar curvature of
$(M,g)$, $\kappa\in\mathbb{R}$, and $2^{\ast}:=\frac{2m}{m-2}$ is the critical
Sobolev exponent. Here we shall always assume that $\kappa>0.$

This problem was completely solved by the combined efforts of Yamabe \cite{y},
Trudinger \cite{t}, Aubin \cite{a1} and Schoen \cite{sch}. A detailed
discussion may be found in \cite{a2, lp}. Obata \cite{o} showed that for an
Einstein metric the solution to the Yamabe problem is unique. On the other
hand, Pollack \cite{po} showed that, if $R_{g}>0,$ then there is a prescribed
number of positive solutions to the Yamabe problem with constant positive
scalar curvature in a conformal class which is arbitrarily close to $g$ in the
$\mathcal{C}^{0}$-topology. Compactness of the set of positive solutions was
established by Khuri, Marques and Schoen \cite{kms} if $(M,g)$ is not
conformally equivalent to the standard sphere and $\dim M\leq24.$ On the other
hand, if $M\geq25,$ Brendle \cite{b} and Brendle and Marques \cite{bm} showed
that the set of positive solutions is not compact. The equivariant Yamabe
problem was studied by Hebey and Vaugon. They showed in \cite{hv} that for any
subgroup $\Gamma$ of the group of isometries of $(M,g)$ there exists a
positive least energy $\Gamma$-invariant solution to the Yamabe problem.

If $u$ is a nodal solution to problem (\ref{yamabe}), i.e., if $u$ changes
sign, then\textit{ }$\hat{g}=\left\vert u\right\vert ^{4/(m-2)}g$ is not a
metric, as $\hat{g}$ is not smooth and it vanishes on the set of zeroes of
$u$. Ammann and Humbert called $\hat{g}$ a \textit{generalized metric}. In
\cite{ah} they showed that, if the Yamabe invariant of $(M,g)$ is nonnegative,
$(M,g)$ is not locally conformaly flat and $\dim M\geq11$, then there exists a
minimal energy nodal solution to (\ref{yamabe}). El Sayed considered the case
where the Yamabe invariant of $(M,g)$ is strictly negative in \cite{e}. Nodal
solutions to (\ref{yamabe}) on some product manifolds have been obtained,
e.g., in \cite{pet, hen}.

On the other hand, multiplicity of nodal solutions to the Yamabe problem
(\ref{yamabe}) is, largely, an open question. In a classical paper \cite{di},
W.Y. Ding established the existence of infinitely many nodal solutions to this
problem on the standard sphere $\mathbb{S}^{m}.$ He took advantage of the fact
that $\mathbb{S}^{m}$ is invariant under the action of isometry groups whose
orbits are positive dimensional.

In this paper we shall study the effect of the isometries of $M$ on the
multiplicity of nodal solutions to Yamabe-type equations. Our framework is as follows.

Let $(M,g)$ be a closed Riemannian manifold of dimension $m\geq3$ and $\Gamma$
be a closed subgroup of the group of isometries $\text{Isom}_{g}(M)$ of
$(M,g).$ As usual, \textit{closed} means compact and without boundary. We
denote by $\Gamma p:=\{\gamma p:\gamma\in\Gamma\}$ the $\Gamma$-orbit of a
point $p\in M$ and by $\#\Gamma p$ its cardinality. Recall that a subset $X$
of $M$ is said to be $\Gamma$-invariant if $\Gamma x\subset X$ for every $x\in
X,$ and a function $f:X\rightarrow\mathbb{R}$ is $\Gamma$-invariant if it is
constant on each orbit $\Gamma x$ of $X.$

We consider the Yamabe-type problem
\begin{equation}
\left\{
\begin{tabular}
[c]{l}%
$-\,\text{div}_{g}(a\nabla_{g}u)+bu=c|u|^{2^{\ast}-2}u,$\\
$u\in H_{g}^{1}(M)^{\Gamma},$%
\end{tabular}
\right.  \label{Equation : generalized Yamabe PDE}%
\end{equation}
where $a,b,c\in\mathcal{C}^{\infty}(M)$ are $\Gamma$-invariant functions, $a$
and $c$ are positive on $M$ and the operator $-\,$div$_{g}(a\nabla_{g})+b$ is
coercive on the space%
\[
H_{g}^{1}(M)^{\Gamma}:=\{u\in H_{g}^{1}(M):u\text{ is }\Gamma\text{-invariant}%
\}.
\]
If $a\equiv1$, $b=\mathfrak{c}_{m}R_{g}$ and $c\equiv\kappa$ is constant, this
is the Yamabe problem (\ref{yamabe}). In this case we shall always assume that
$\kappa>0$ and that the Yamabe operator $\Delta_{g}+\mathfrak{c}_{m}R_{g}$ is
coercive on $H_{g}^{1}(M)^{\Gamma}.$

We will prove the following result.

\begin{theorem}
\label{Theorem : Main1}If $-\,$div$_{g}(a\nabla_{g})+b$ is coercive on
$H_{g}^{1}(M)^{\Gamma}$ and $1\leq\dim(\Gamma p)<m$ for every $p\in M$, then
problem \eqref{Equation : generalized Yamabe PDE} has at least one positive
and infinitely many nodal $\Gamma$-invariant solutions.
\end{theorem}

A special case is the following multiplicity result for the Yamabe problem
(\ref{yamabe}).

\begin{corollary}
\label{Corollary : Theorem Main 1}If $\Delta_{g}+\mathfrak{c}_{m}R_{g}$ is
coercive on $H_{g}^{1}(M)^{\Gamma}$ and $1\leq\dim(\Gamma p)<m$ for all $p\in
M,$ then the Yamabe problem \emph{(\ref{yamabe})} has infinitely many $\Gamma
$-invariant nodal solutions.
\end{corollary}

The standard sphere $(\mathbb{S}^{m},g_{0})$ is invariant under the action of
the group $O(k)\times O(n)$ with $k+n=m+1,$ and this action has positive
dimensional orbits if $k,n\geq2.$ So Corollary
\ref{Corollary : Theorem Main 1}\ can be seen as a generalization of Ding's
result \cite{di}. One may also consider the action of $\mathbb{S}^{1}$ on the
standard sphere $\mathbb{S}^{2k+1}\subset\mathbb{C}^{k}$ given by complex
multiplication on each complex coordinate. In this case, every orbit has
dimension one.

Further examples are obtained as follows: if $\Gamma$ is a closed subgroup of
the group of isometries of $(\mathbb{S}^{m},g_{0}),\ (N,h)$ is a closed
Riemannian manifold of dimension $n$ and $f\in\mathcal{C}^{\infty}(N)$ is a
positive function, then $\Gamma$ acts on the warped product $N\times
_{f}\mathbb{S}^{m}=(N\times\mathbb{S}^{m},h+f^{2}g_{0})$ in the obvious way.
So, if $m+n\geq3$, $\Delta_{g}+\mathfrak{c}_{m}R_{g}$ is coercive on
$H_{h+f^{2}g_{0}}^{1}(N\times_{f}\mathbb{S}^{2k+1})^{\Gamma}$ and every
$\Gamma$-orbit of $\mathbb{S}^{m}$ is positive dimensional, then the Yamabe
problem (\ref{yamabe}) has infinitely many $\Gamma$-invariant nodal solutions
on $N\times_{f}\mathbb{S}^{m}$. This extends Theorem 1.2 in \cite{pet}.

Next, we study a case in which $M$ is allowed to have finite $\Gamma$-orbits.
We consider the following setting:

Let $M$ be a closed smooth $m$-dimensional manifold and $a,b,c\in
\mathcal{C}^{\infty}(M)$ be such that $a$ and $c$ are positive on $M$. We fix
an open subset $\Omega$ of $M,$ a Riemannian metric $h$ on $\Omega$ and a
compact subgroup $\Lambda$ of Isom$_{h}(\Omega)$ such that $\dim(\Lambda p)<m$
for all $p\in\Omega$, the restrictions of $a,b,c$ to $\Omega$ are $\Lambda
$-invariant and the operator $-\,$div$_{g}(a\nabla_{g})+b$ is coercive on the
space $\mathcal{C}_{c}^{\infty}(\Omega)^{\Lambda}$ of smooth $\Lambda
$-invariant functions with compact support in $\Omega.$\ Under these
assumptions, we will prove the following multiplicity result.

\begin{theorem}
\label{Theorem : Main2}There exists an increasing sequence $(\ell_{k})$ of
positive real numbers, depending only on $(\Omega,h),$ $a,b,c$ and $\Lambda$,
with the following property: For any Riemanniann metric $g$ on $M$ and any
closed subgroup $\Gamma$ of \emph{Isom}$_{g}(M)$ which satisfy

\begin{enumerate}
\item $g=h$ in $\Omega$;

\item $\Gamma$ is a subgroup of $\Lambda$ and $a,b,c$ are $\Gamma$-invariant;

\item $-\,$\emph{div}$_{g}(a\nabla_{g})+b$ is coercive on $H_{g}%
^{1}(M)^{\Gamma}$;\smallskip

\item $\min\limits_{p\in M}\frac{a(p)^{m/2}\,\#\Gamma p}{c(p)^{\frac{m-2}{2}}%
}>\ell_{k}$;\smallskip
\end{enumerate}

\noindent problem \eqref{Equation : generalized Yamabe PDE} has at least $k$
pairs of $\Gamma$-invariant solutions $\pm u_{1},\ldots,\pm u_{k}$ such that
$u_{1}$ is positive, $u_{2},\ldots,u_{k}$ change sign, and
\begin{equation}
\int_{M}c\left\vert u_{j}\right\vert ^{2^{\ast}}dV_{g}\leq\ell_{j}%
S^{m/2}\text{\qquad for every }j=1,\ldots,k,
\label{Theorem : Main2 inequality}%
\end{equation}
where $S$ is the best Sobolev constant for the embedding $D^{1,2}%
(\mathbb{R}^{m})\hookrightarrow L^{2^{\ast}}(\mathbb{R}^{m})$.
\end{theorem}

Theorem \ref{Theorem : Main2}\ asserts the existence of a prescribed number of
nodal solutions to problem \eqref{Equation : generalized Yamabe PDE} if there
is a Riemannian metric on $M,$ which extends the given Riemannian metric on
$\Omega,$ for which some group of isometries has large enough orbits.

Nodal solutions to Yamabe-type equations have been exhibited, e.g., in
\cite{dj, hol, v}. If $m\geq4,$ $a=c\equiv1$ and $\Delta_{g}+b$ is coercive,
V\'{e}tois showed that problem \eqref{Equation : generalized Yamabe PDE} has
at least $\frac{m+2}{2}$\ solutions provided that $b(p_{0})<\mathfrak{c}%
_{m}R_{g}(p_{0})$ at some point $p_{0}\in M$ \cite{v}. This last assumption
excludes the Yamabe problem (\ref{yamabe}). Also, nothing is said about the
sign of the solutions, except for the cases when the positive solution is
known to be unique.

In contrast, Theorem \ref{Theorem : Main2} does apply to the Yamabe problem.
However, property (4) requires that the group $\Lambda$ has large enough
subgroups. The group $\mathbb{S}^{1},$ for example, has this property. This
allows us to derive a multiplicity result for the Yamabe problem
(\ref{yamabe}) in the following setting.

Let $(M,h)$ be a closed Riemannian manifold on which $\mathbb{S}^{1}$ acts
freely and isometrically, such that $\Delta_{h}+\mathfrak{c}_{m}R_{h}$ is
coercive in $H_{h}^{1}(M).$ Fix an open $\mathbb{S}^{1}$-invariant subset
$\Omega$ of $M$ such that $R_{h}>0$ on $M\smallsetminus\Omega.$ Set
$\Gamma_{n}:=\{\mathrm{e}^{2\pi\mathrm{i}j/n}:j=0,...,n-1\}$. Then, the
following statement holds true.

\begin{corollary}
\label{cor:main}There exist a sequence $(\ell_{k})$ in $(0,\infty)$ and an
open neighborhood $\mathcal{O}$ of $h$ in the space of Riemannian metrics on
$M$ with the $\mathcal{C}^{0}$-topology, with the following property: for
every $g\in\mathcal{O}$ such that $g=h$ in $\Omega$ and$\ \Gamma_{n}\subset
\,$\emph{Isom}$_{g}(M)$ for some $n>\kappa^{(m-2)/2}\ell_{k},$ the Yamabe
problem \emph{(\ref{yamabe})} has at least $k$ pairs of $\Gamma_{n}$-invariant
solutions $\pm u_{1},\ldots,\pm u_{k}$ such that $u_{1}$ is positive,
$u_{2},\ldots,u_{k}$ change sign, and
\[
\int_{M}\left\vert u_{j}\right\vert ^{2^{\ast}}dV_{g}\leq\kappa^{-1}\ell
_{j}S^{m/2}\text{\qquad for every }j=1,\ldots,k.
\]

\end{corollary}

For instance, we may take $\Omega$ to be the complement of a closed tubular
neighborhood of an $\mathbb{S}^{1}$-orbit in $(M,h)$ on which $R_{h}>0.$ Then
$M\smallsetminus\Omega$ is $\mathbb{S}^{1}$-diffeomorphic to $\mathbb{S}%
^{1}\times\mathbb{B}^{m-1},$ where $\mathbb{B}^{m-1}$ is the closed unit ball
in $\mathbb{R}^{m-1}$. We choose $n>\kappa^{(m-2)/2}\ell_{k}$. Then, if we
modify the metric in the interior of the piece of $M\smallsetminus\Omega$
which corresponds to $\{\mathrm{e}^{2\pi\mathrm{i}\vartheta/n}:0\leq
\vartheta\leq1\}\times\mathbb{B}^{m-1}$ and translate this modification to
each of the pieces corresponding to $\{\mathrm{e}^{2\pi\mathrm{i}\vartheta
/n}:j-1\leq\vartheta\leq j\}\times\mathbb{B}^{m-1},$ $j=2,...,n,$ we obtain a
metric $g$ on $M$ such that $g=h$ in $\Omega$ and$\ \Gamma_{n}\subset
\,$Isom$_{g}(M).$ If $g$ is chosen to be close enough to $h$, then the
previous corollary asserts the existence of $k$ pairs of solutions to the
Yamabe problem (\ref{yamabe}). This way we obtain many examples of Riemannian
manifolds with finite symmetries which admit a prescribed number of nodal
solutions to the Yamabe problem.

We would like to mention that existence and multiplicity of positive and nodal
solutions are also available for some perturbations of the Yamabe problem;
see, e.g., \cite{mpv, rv} and the references therein.

Finally, we wish to stress that, even though the Yamabe invariant is always
attained, problem \eqref{Equation : generalized Yamabe PDE} need not have a
ground state solution, as the following example shows. So a solution cannot
always be obtained by minimization.

\begin{proposition}
\label{prop:nonexistence}If $(\mathbb{S}^{m},g_{0})$ is the standard sphere
and $b\in\mathcal{C}^{\infty}(\mathbb{S}^{m})$ is such that $b\geq
\mathfrak{c}_{m}R_{g_{0}}=\frac{m(m-2)}{4}$ and $b\not \equiv \mathfrak{c}%
_{m}R_{g_{0}}$, then the equation
\[
\Delta_{g_{0}}u+bu=|u|^{2^{\ast}-2}u,\qquad u\in\mathcal{C}^{\infty
}(\mathbb{S}^{m})\text{,}%
\]
does not admit a ground state solution, i.e.,%
\[
\inf_{\substack{u\in\mathcal{C}^{\infty}(\mathbb{S}^{m})\\u\neq0}}\frac
{\int_{\mathbb{S}^{m}}\left[  |\nabla_{g_{0}}u|_{g_{0}}^{2}+bu^{2}\right]
dV_{g_{0}}}{\left(  \int_{\mathbb{S}^{m}}|u|^{2^{\ast}}dV_{g_{0}}\right)
^{2/2^{\ast}}}%
\]
is not attained.
\end{proposition}

Theorems \ref{Theorem : Main1} and \ref{Theorem : Main2} and Corollary
\ref{cor:main} will be proved in Section \ref{sec:main}. Their proof follows
some ideas introduced in \cite{cf1}, where a result similar to Theorem
\ref{Theorem : Main2}, in a bounded domain of $\mathbb{R}^{m},$ is
established. The proof is based on a compactness result and a variational
principle for nodal solutions which are proved in Sections
\ref{Section : Compactness} and \ref{sec:vp} respectively. Proposition
\ref{prop:nonexistence} is proved in Section \ref{sec:nonexistence}.

\section{Proof of the main results}

\label{sec:main}Let $(M,g)$ be a closed Riemannian manifold of dimension
$m\geq3$, $\Gamma$ be a closed subgroup of $\text{Isom}_{g}(M),$ and
$a,b,c\in\mathcal{C}^{\infty}(M)$ be $\Gamma$-invariant functions. We will
assume throughout this section that $a>0,$ $c>0$ and that the operator
$-\,$div$_{g}(a\nabla_{g})+b$ is coercive on the space $H_{g}^{1}(M)^{\Gamma
}:=\{u\in H_{g}^{1}(M):u$ is $\Gamma$-invariant$\}.$ Then,%
\[
\langle u,v\rangle_{g,a,b}:=\int_{M}\left[  a\langle\nabla_{g}u,\nabla
_{g}v\rangle_{g}+buv\right]  dV_{g}%
\]
is an interior product in $H_{g}^{1}(M)^{\Gamma}$ and the induced norm, which
we will denote by $\Vert\cdot\Vert_{g,a,b}$, is equivalent to the standard
norm $\Vert\cdot\Vert_{g}$ in $H_{g}^{1}(M)^{\Gamma}$. Also,
\[
|u|_{g,c,2^{\ast}}:=\left(  \int_{M}c|u|^{2^{\ast}}dV_{g}\right)  ^{1/2^{\ast
}}%
\]
defines a norm in $L_{g}^{2^{\ast}}(M)$ which is equivalent to the standard
norm $|\cdot|_{g,2^{\ast}}$.

By the principle of symmetric criticality \cite{p}, the solutions to problem
\eqref{Equation : generalized Yamabe PDE} are the critical points of the
energy functional
\begin{align*}
J_{g}(u)  &  =\frac{1}{2}\int_{M}\left[  a|\nabla_{g}u|_{g}^{2}+bu^{2}\right]
dV_{g}-\frac{1}{2^{\ast}}\int_{M}c|u|^{2^{\ast}}dV_{g}\\
&  =\frac{1}{2}\Vert u\Vert_{g,a,b}^{2}-\frac{1}{2^{\ast}}|u|_{g,c,2^{\ast}%
}^{2^{\ast}}%
\end{align*}
defined on the space $H_{g}^{1}(M)^{\Gamma}.$ The nontrivial ones lie on the
Nehari manifold%
\begin{equation}
\mathcal{N}_{g}^{\Gamma}:=\{u\in H_{g}^{1}(M)^{\Gamma}:u\neq0,\text{ \ }\Vert
u\Vert_{g,a,b}^{2}=|u|_{g,c,2^{\ast}}^{2^{\ast}}\} \label{eq:nehari}%
\end{equation}
which is of class $\mathcal{C}^{2},$ radially diffeomorphic to the unit sphere
in $H_{g}^{1}(M)^{\Gamma},$ and a natural constraint for $J_{g}.$ Moreover,
for every $u\in H_{g}^{1}(M)^{\Gamma},$ $u\neq0,$
\begin{equation}
u\in\mathcal{N}_{g}^{\Gamma}\iff J_{g}(u)=\max_{t\geq0}J_{g}(tu).
\label{Equation : Maximum lies in Nehari}%
\end{equation}
Set
\[
\tau_{g}^{\Gamma}:=\inf_{\mathcal{N}_{g}^{\Gamma}}J_{g}.
\]
The continuity of the Sobolev embedding $H_{g}^{1}(M)\hookrightarrow
L_{g}^{2^{\ast}}(M)$ implies that $\tau_{g}^{\Gamma}>0.$\

The proofs of Theorems \ref{Theorem : Main1} and \ref{Theorem : Main2} follow
the scheme introduced in \cite{cf1,cf2}. They are based on a compactness
result and a variational principle for nodal solutions, which are stated next.

\begin{definition}
A $\Gamma$-invariant Palais-Smale sequence for the functional $J_{g}$ at the
level $\tau$ is a sequence $(u_{k})$ such that,
\[
u_{k}\in H_{g}^{1}(M)^{\Gamma},\qquad J_{g}(u_{n})\rightarrow\tau,\qquad
J_{g}^{\prime}(u_{k})\rightarrow0\text{ in }\left(  H_{g}^{1}(M)\right)
^{\prime}.
\]
We shall say that $J_{g}$ satisfies condition $(PS)_{\tau}^{\Gamma}$ in
$H_{g}^{1}(M)$ if every $\Gamma$-invariant Palais-Smale sequence for $J_{g}$
at the level $\tau$ contains a subsequence which converges strongly in
$H_{g}^{1}(M).$
\end{definition}

The presence of symmetries allows to increase the lowest level at which this
condition fails. The following result will be proved in Section
\ref{Section : Compactness}.

\begin{theorem}
[Compactness]\label{Theorem : Compactness}The functional $J_{g}$ satisfies
condition $(PS)_{\tau}^{\Gamma}$ in $H_{g}^{1}(M)$ for every
\[
\tau<\left(  \min_{q\in M}\frac{a(q)^{m/2}\,\#\Gamma q}{c(q)^{(m-2)/2}%
}\right)  \frac{1}{m}S^{m/2},
\]
where $S$ is the best Sobolev constant for the embedding $D^{1,2}%
(\mathbb{R}^{m})\hookrightarrow L^{2^{\ast}}(\mathbb{R}^{m})$.
\end{theorem}

If all $\Gamma$-orbits in $M$ have positive dimension, this result says that
$J_{g}$ satisfies $(PS)_{\tau}^{\Gamma}$ for every $\tau\in\mathbb{R}.$ This
can also be deduced from the compactness of the Sobolev embedding $H_{g}%
^{1}(M)^{\Gamma}\hookrightarrow L_{g}^{2^{\ast}}(M)$ which was proved by Hebey
and Vaugon in \cite{hv1}. However, this embedding is not longer compact when
$M$ contains a finite orbit, as in the situation considered in Theorem
\ref{Theorem : Main2}.

The variational principle that we will use is the following one. It will be
proved in Section \ref{sec:vp}.

\begin{theorem}
[Sign-changing critical points]\label{thm: Variational principle}Let $W$ be a
nontrivial finite dimensional subspace of $H_{g}^{1}(M)^{\Gamma}$. If $J_{g}$
satisfies $(PS)_{\tau}^{\Gamma}$ in $H_{g}^{1}(M)$ for every $\tau\leq\sup
_{W}J_{g}$, then $J_{g}$ has at least one positive critical point $u_{1}$ and
$\dim W-1$ pairs of nodal critical points $\pm u_{2},...,\pm u_{k}$ in
$H_{g}^{1}(M)^{\Gamma}$ such that $J_{g}(u_{1})=\tau_{g}^{\Gamma}$ and
$J_{g}(u_{i})\leq\sup_{W}J_{g}$ for $i=1,...,k.$
\end{theorem}

For the proof of Theorems \ref{Theorem : Main1} and \ref{Theorem : Main2} we
also need the following well known result. Recall that the $\Gamma$-orbit
space of a $\Gamma$-invariant subset $X$ of $M$ is the set $X/\Gamma$ of all
$\Gamma$-orbits in $X,$ with the quotient topology. The $\Gamma$-isotropy
subgroup of a point $p\in M$ is defined as $\Gamma_{p}:=\{\gamma\in
\Gamma:\gamma p=p\}$. The $\Gamma$-orbit $\Gamma p$ of $p$ is $\Gamma
$-diffeomorphic to the homogeneous space $\Gamma/\Gamma_{p}$. Isotropy
subgroups satisfy $\Gamma_{\gamma p}=\gamma\Gamma_{p}\gamma^{-1}$. Thus, every
subgroup of $\Gamma$ which is conjugate to an isotropy subgroup is also an
isotropy subgroup; see, e.g., \cite{bre,td}. We denote by $(H)$ the conjugacy
class of a subgroup $H$ of $\Gamma.$

\begin{theorem}
\label{thm:principal_orbit}Let $M$ be a smooth connected manifold with a
smooth action of a compact Lie group $\Gamma$. Then there exists a closed
subgroup $H$ of $\Gamma$ such that the set $M_{(H)}:=\{p\in M:(\Gamma
_{p})=(H)\}$ is open and dense in $M.$ Its orbit space $M_{(H)}/\Gamma$ is a
smooth manifold of dimension $m-\dim(\Gamma/H),$ and the quotient map
$M_{(H)}\rightarrow M_{(H)}/\Gamma$ is a fiber bundle with fiber $\Gamma/H$.
\end{theorem}

\begin{proof}
See Theorems IV.3.1, IV.3.3 and IV.3.8 in \cite{bre}, or Theorem I.5.11 in
\cite{td}.
\end{proof}

Next, we derive our main results from the previous three theorems.

\begin{proof}
[\textbf{Proof of Theorem \ref{Theorem : Main1}}]By Theorem
\ref{thm:principal_orbit}, $M$ contains an open dense subset $\Omega:=M_{(H)}$
such that the $\Gamma$-orbit of each point $p\in\Omega$ is $\Gamma
$-diffeomorphic to $\Gamma/H$ for some fixed closed subgroup $H$ of $\Gamma.$
Moreover, $\Gamma p$ has a $\Gamma$-invariant neighborhood $\Omega_{p}$
contained in $\Omega$ which is $\Gamma$-diffeomorphic to $\mathbb{B}%
\times\Gamma/H,$ where $\mathbb{B}$ is the euclidean unit ball of dimension
$m-\dim(\Gamma p)$. Since we are assuming that $\dim(\Gamma p)<m,$ for any
given $k\in\mathbb{N}$ we may choose $k$ different $\Gamma$-orbits $\Gamma
p_{1},\ldots,\Gamma p_{k}\subset\Omega$ and $\Gamma$-invariant neighborhoods
$\Omega_{p_{i}}$ as before, with $\Omega_{p_{i}}\cap\Omega_{p_{j}}=\emptyset$
if $i\neq j.$ Then, we can choose a $\Gamma$-invariant function $\omega_{i}%
\in\mathcal{C}_{c}^{\infty}(\Omega_{p_{i}})$ for each $i=1,\ldots,k$.

Let $W:=\text{span}\{\omega_{1},\dots,\omega_{k}\}$ be the linear subspace of
$H_{g}^{1}(M)^{\Gamma}$ spanned by $\{\omega_{1},\ldots,\omega_{k}\}$. As
$\omega_{i}$ and $\omega_{j}$ have disjoint supports for $i\neq j,$ the set
$\{\omega_{1},\dots,\omega_{k}\}$ is orthogonal in $H_{g}^{1}(M)^{\Gamma}$.
Hence, $\dim W=k.$ On the other hand, as $\dim(\Gamma p)\geq1,$ we have that
$\#\Gamma p=\infty$ for every $p\in M.$ So, by Theorem
\ref{Theorem : Compactness}, $J_{g}$ satisfies $(PS)_{\tau}^{\Gamma}$ in
$H_{g}^{1}(M)$ for every $\tau\in\mathbb{R}$. Therefore, Theorem
\ref{thm: Variational principle} yields at least one positive and $k-1$ nodal
$\Gamma$-invariant solutions to problem
\eqref{Equation : generalized Yamabe PDE}. As $k\in\mathbb{N}$ is arbitrary,
we conclude that there are infinitely many nodal solutions.
\end{proof}

\begin{proof}
[\textbf{Proof of Theorem \ref{Theorem : Main2}}]By Theorem
\ref{thm:principal_orbit}, after replacing $\Omega$ by a $\Lambda$-invariant
open subset of it, if necessary, we may assume that $\Lambda p$ is $\Lambda
$-diffeomorphic to $\Lambda/H$ for every $p\in\Omega$ and some fixed subgroup
$H$ of $\Lambda$. Let $\mathcal{P}_{1}(\Omega)$ be the family of all nonempty
$\Lambda$-invariant open subsets of $\Omega$ and, for each $\tilde{\Omega}%
\in\mathcal{P}_{1}(\Omega),$ set%
\[
\mathcal{D}(\tilde{\Omega}):=\{\varphi\in\mathcal{C}_{c}^{\infty}%
(\tilde{\Omega}):\varphi\text{ is }\Lambda\text{-invariant, }\varphi
\neq0,\text{ }\Vert\varphi\Vert_{h,a,b}^{2}=|\varphi|_{h,c,2^{\ast}}^{2^{\ast
}}\}.
\]
For each $k\in\mathbb{N}$ let
\[
\mathcal{P}_{k}(\Omega):=\{(\Omega_{1},\ldots,\Omega_{k}):\Omega_{i}%
\in\mathcal{P}_{1}(\Omega),\text{ \ }\Omega_{i}\cap\Omega_{j}=\emptyset\text{
if }i\neq j\}.
\]
Arguing as in the proof of Theorem \ref{Theorem : Main1} we see that
$\mathcal{P}_{k}(\Omega)\neq\emptyset$ and $\mathcal{D}(\tilde{\Omega}%
)\neq\emptyset.$ Set%
\[
\tau_{k}:=\inf\left\{  \sum_{i=1}^{k}\frac{1}{m}\Vert\varphi_{i}\Vert
_{h,a,b}^{2}:\varphi_{i}\in\mathcal{D}(\Omega_{i}),\text{ \ }(\Omega
_{1},\ldots,\Omega_{k})\in\mathcal{P}_{k}(\Omega)\right\}  ,
\]
and define%
\[
\ell_{k}:=\left(  \frac{1}{m}S^{m/2}\right)  ^{-1}\tau_{k}.
\]
Next, we show that the sequence $(\ell_{k})$ has the desired property.

Fix $k\in\mathbb{N}$, and let $(M,g)$ be a Riemanniann manifold and $\Gamma$
be a closed subgroup of Isom$_{g}(M)$ which satisfy (1)-(4). As $g=h$ in
$\Omega$ and $\Gamma$ is a subgroup of $\Lambda,$ extending $\varphi\in
C_{c}^{\infty}(\tilde{\Omega})$ by zero outside $\tilde{\Omega}$, we have that
$\mathcal{D}(\tilde{\Omega})\subset\mathcal{N}_{g}^{\Gamma}$ for every
$\tilde{\Omega}\in\mathcal{P}_{1}(\Omega)$, $J_{g}(\varphi)=\frac{1}{m}%
\Vert\varphi\Vert_{h,a,b}^{2}$ for every $\varphi\in\mathcal{D}(\tilde{\Omega
})$ and $\tau_{1}\geq\tau_{g}^{\Gamma}>0$. Since we are assuming that%
\[
\ell_{a,c}^{\Gamma}:=\min\limits_{p\in M}\frac{a(p)^{m/2}\,\#\Gamma
p}{c(p)^{(m-2)/2}}>\ell_{k},
\]
we may choose $\varepsilon\in(0,\tau_{1})$ such that $\tau_{k}+\varepsilon
<\ell_{a,c}^{\Gamma}(\frac{1}{m}S^{m/2})$. Then, by definition of $\tau_{k}$,
there exist $(\Omega_{1},\ldots,\Omega_{k})\in\mathcal{P}_{k}(\Omega)$ and
$\omega_{i}\in\mathcal{D}(\Omega_{i}),$ such that
\[
\tau_{k}\leq%
{\textstyle\sum\limits_{i=1}^{k}}
J_{g}(\omega_{i})<\tau_{k}+\varepsilon.
\]
For each $n=1,\ldots,k$ set $W_{n}:=\text{span}\{\omega_{1}\ldots,\omega
_{n}\}$. As $\omega_{i}$ and $\omega_{j}$ have disjoint supports for $i\neq
j,$ the set $\{\omega_{1},\dots,\omega_{k}\}$ is orthogonal in $H_{g}%
^{1}(M)^{\Gamma}$. Hence, $\dim W_{n}=n.$ Moreover, if $u\in W_{n},$
$u=\sum_{i=1}^{n}t_{i}\omega_{i}$, then
\eqref{Equation : Maximum lies in Nehari} yields
\[
J_{g}(u)=%
{\textstyle\sum\limits_{i=1}^{n}}
J_{g}(t_{i}\omega_{i})\leq%
{\textstyle\sum\limits_{i=1}^{n}}
J_{g}(\omega_{i})<\tau_{k}+\varepsilon.
\]
Therefore,
\[
\sigma_{n}:=\sup_{W_{n}}J_{g}\leq\tau_{k}+\varepsilon<\ell_{a,c}^{\Gamma
}(\frac{1}{m}S^{m/2}).
\]
So Theorems \ref{Theorem : Compactness} and \ref{thm: Variational principle}
yield a positive critical point $u_{1}$ and $n-1$ pairs of sign changing
critical points $\pm u_{n,2},\ldots,\pm u_{n,n}$ of $J_{g}$\ in $H_{g}%
^{1}(M)^{\Gamma}$ such that $J_{g}(u_{1})=\tau_{g}^{\Gamma}$ and
\[
J_{g}(u_{n,j})\leq\sigma_{n}\text{\qquad for all \ }j=2,\ldots,n.
\]

Now, for each $2\leq n\leq k$, we inductively choose $u_{n}\in\{u_{n,2}%
,\ldots,u_{n,n}\}$ such that $u_{n}\neq u_{j}$ for all $1\leq j<n$. In order
to show that the $u_{j}$'s may be suitable chosen to satisfy
\eqref{Theorem : Main2 inequality}, we need the following inequalities.
Observe that $\tau_{1}\leq J_{g}(\omega_{i})$ for every $i=1,\ldots,k$.
Consequently, for each $2\leq n\leq k$ we obtain
\[
\sigma_{n}+(k-n)\tau_{1}\leq%
{\textstyle\sum\limits_{i=1}^{n}}
J_{g}(\omega_{i})+%
{\textstyle\sum\limits_{i=n+1}^{k}}
J_{g}(\omega_{i})<\tau_{k}+\varepsilon.
\]
As $\varepsilon\in(0,\tau_{1})$ we conclude that
\[
J_{g}(u_{n})\leq\sigma_{n}<\tau_{k}\text{ \ if }n<k\qquad\text{and}\qquad
J_{g}(u_{k})\leq\sigma_{k}<\tau_{k}+\varepsilon.
\]
With these inequalities, the argument in the last two steps of the proof of
Theorem 2.2 in \cite{cf2} goes through to show that the $u_{j}^{\prime}s$ may
be chosen so that \eqref{Theorem : Main2 inequality} is satisfied.
\end{proof}

\begin{proof}
[\textbf{Proof of Corollary \ref{cor:main}}]Let $\mathfrak{M}$ be the space of
Riemannian metrics on $M$ with the distance induced by the $\mathcal{C}^{0}%
$-norm in the space of covariant $2$-tensor fields $\tau$ on $M,$ taken with
respect to the fixed metric $h$, i.e.%
\[
\left\Vert \tau\right\Vert _{\mathcal{C}^{0}}:=\max_{p\in M}\,\max_{X,Y\in
T_{p}M\smallsetminus\{0\}}\frac{\left\vert \tau(X,Y)\right\vert }{\left\vert
X\right\vert _{h}\left\vert Y\right\vert _{h}}.
\]
As the functions $\mathfrak{M}\rightarrow\mathcal{C}^{0}(M)$ given by
$g\mapsto R_{g}$ and $g\mapsto\sqrt{\left\vert g\right\vert }$ are continuous,
where $\left\vert g\right\vert :=\det(g),$ the sets%
\begin{align*}
\mathcal{O}_{1}  &  :=\left\{  g\in\mathfrak{M}:\frac{1}{2}R_{h}%
(p)<R_{g}(p)<2R_{h}(p)\text{ \ }\forall p\in M\smallsetminus\Omega\right\}
,\\
\mathcal{O}_{2}  &  :=\left\{  g\in\mathfrak{M}:\frac{1}{2}\sqrt{\left\vert
h\right\vert (p)}<\sqrt{\left\vert g\right\vert (p)}<2\sqrt{\left\vert
h\right\vert (p)}\text{ \ }\forall p\in M\smallsetminus\Omega\right\}  ,
\end{align*}
are open neighborhoods of $h$ in $\mathfrak{M}$. Moreover, since%
\[
|\nabla_{g}u(p)|_{g}=\max_{X\in T_{p}M\smallsetminus\{0\}}\frac{\left\vert
duX\right\vert }{\left\vert X\right\vert _{g}},
\]
for every $u\in\mathcal{C}^{\infty}(M)$ we have that%
\[
\frac{1}{2}|\nabla_{h}u|_{h}^{2}\leq|\nabla_{g}u|_{g}^{2}\leq2|\nabla
_{h}u|_{h}^{2}\text{\qquad if }\left\Vert g-h\right\Vert _{\mathcal{C}^{0}%
}<\frac{1}{2}.
\]
Set $\mathcal{O}:=\{g\in\mathfrak{M}:\left\Vert g-h\right\Vert _{\mathcal{C}%
^{0}}<\frac{1}{2}\}\cap\mathcal{O}_{1}\cap\mathcal{O}_{2}.$ Then there are
positive constants $C_{1}\leq1$ and $C_{2}\geq1$ such that, for every
$g\in\mathcal{O}$ and $u\in\mathcal{C}^{\infty}(M),$%
\begin{align*}
\int_{M\smallsetminus\Omega}\left[  |\nabla_{g}u|_{g}^{2}+\mathfrak{c}%
_{m}R_{g}u^{2}\right]  dV_{g}  &  \geq C_{1}\int_{M\smallsetminus\Omega
}\left[  |\nabla_{h}u|_{h}^{2}+\mathfrak{c}_{m}R_{h}u^{2}\right]  dV_{h},\\
\int_{M\smallsetminus\Omega}\left[  |\nabla_{g}u|_{g}^{2}+u^{2}\right]
dV_{g}  &  \leq C_{2}\int_{M\smallsetminus\Omega}\left[  |\nabla_{h}u|_{h}%
^{2}+u^{2}\right]  dV_{h}.
\end{align*}
Therefore, if $g\in\mathcal{O}$ and $g=h$ in $\Omega$, we have that
\[
\frac{\int_{M}\left[  |\nabla_{g}u|_{g}^{2}+\mathfrak{c}_{m}R_{g}u^{2}\right]
dV_{g}}{\int_{M}\left[  |\nabla_{g}u|_{g}^{2}+u^{2}\right]  dV_{g}}\geq
\frac{C_{1}\int_{M}\left[  |\nabla_{h}u|_{h}^{2}+\mathfrak{c}_{m}R_{h}%
u^{2}\right]  dV_{h}}{C_{2}\int_{M}\left[  |\nabla_{h}u|_{h}^{2}+u^{2}\right]
dV_{h}}%
\]
for every $u\in\mathcal{C}^{\infty}(M).$ As $\Delta_{h}+\mathfrak{c}_{m}R_{h}$
is coercive in $H_{h}^{1}(M)$, this inequality implies that $\Delta
_{g}+\mathfrak{c}_{m}R_{g}$ is coercive in $H_{g}^{1}(M).$

Set $(\Omega,h)$ as given, $\Lambda=\mathbb{S}^{1},$ $a\equiv1,$
$b=\mathfrak{c}_{m}R_{g}$ and $c\equiv\kappa.$ Then, if $g\in\mathcal{O}$ is
such that $g=h$ in $\Omega$ and$\ \Gamma_{n}\subset\,$\emph{Isom}$_{g}(M)$ for
some $n>\kappa^{(m-2)/2}\ell_{k},$ these data satisfy assumptions (1)-(4) in
Theorem \ref{Theorem : Main2}, and the conclusion follows.
\end{proof}

\section{Nonexistence of ground state solutions}

\label{sec:nonexistence}In this section we prove Proposition
\ref{prop:nonexistence}.

If $h$ and $g=\varphi^{2^{\ast}-2}h,$ with $\varphi\in\mathcal{C}^{\infty
}(M),$ $\varphi>0,$ are two conformally equivalent Riemannian metrics on an
$m$-dimensional manifold $M,$ the scalar curvatures $R_{h}$ and $R_{g}$ are
related by the equation%
\begin{equation}
\Delta_{h}\varphi+\mathfrak{c}_{m}R_{h}\varphi=\mathfrak{c}_{m}R_{g}%
\varphi^{2^{\ast}-1}. \label{eq:ne1}%
\end{equation}
Let $v=\varphi u\in\mathcal{C}^{\infty}(M).$ An easy computation shows that
\[
\Delta_{g}u=\varphi^{-2^{\ast}}\left(  \varphi\Delta_{h}v-v\Delta_{h}%
\varphi\right)
\]
and, combining this identity with (\ref{eq:ne1}), we obtain that%
\begin{equation}
\Delta_{g}u+\mathfrak{c}_{m}R_{g}u=\varphi^{1-2^{\ast}}\left(  \Delta
_{h}v+\mathfrak{c}_{m}R_{h}v\right)  . \label{eq:ne2}%
\end{equation}

Let $(\mathbb{S}^{m},g_{0})$ be the standard sphere and $b\in\mathcal{C}%
^{\infty}(\mathbb{S}^{m})$ be such that $b\geq\mathfrak{c}_{m}R_{g_{0}}%
=\frac{m(m-2)}{4}$ and $b\not \equiv \mathfrak{c}_{m}R_{g_{0}}$. Let
$p\in\mathbb{S}^{m}$ be the north pole and $\sigma:\mathbb{S}^{m}%
\smallsetminus\{p\}\rightarrow\mathbb{R}^{m}$ be the stereographic projection.
$\sigma$ is a conformal diffeomorphism and the coordinates of standard metric
$g_{0}$ given by the chart $\sigma^{-1}:\mathbb{R}^{m}\rightarrow
\mathbb{S}^{m}\smallsetminus\{p\}$ are $(g_{0})_{ij}=\varphi^{2^{\ast}%
-2}\delta_{ij},$ where
\[
\varphi(x):=\left(  \frac{2}{1+\left\vert x\right\vert ^{2}}\right)
^{(m-2)/2}.
\]
Set $\widetilde{b}:=\varphi^{2^{\ast}-2}\left(  b\circ\sigma^{-1}%
-\mathfrak{c}_{m}R_{g_{0}}\right)  $ and, for $u\in\mathcal{C}^{\infty
}(\mathbb{S}^{m}),$ set $v=\varphi(u\circ\sigma^{-1}).$ As $dV_{g_{0}}%
=\varphi^{2^{\ast}}dx,$ using (\ref{eq:ne2}) we obtain that%
\begin{align*}
\int_{\mathbb{S}^{m}}\left[  |\nabla_{g_{0}}u|_{g_{0}}^{2}+\mathfrak{c}%
_{m}R_{g_{0}}u^{2}\right]  dV_{g_{0}}  &  =\int_{\mathbb{R}^{m}}\left\vert
\nabla v\right\vert ^{2}dx,\\
\int_{\mathbb{S}^{m}}(b-\mathfrak{c}_{m}R_{g_{0}})u^{2}dV_{g}  &
=\int_{\mathbb{R}^{m}}\widetilde{b}v^{2}dx,\\
\int_{\mathbb{S}^{m}}\left\vert u\right\vert ^{2^{\ast}}dV_{g_{0}}  &
=\int_{\mathbb{R}^{m}}\left\vert v\right\vert ^{2^{\ast}}dx.
\end{align*}
Hence,%
\[
\inf_{\substack{u\in\mathcal{C}^{\infty}(\mathbb{S}^{m})\\u\neq0}}\frac
{\int_{\mathbb{S}^{m}}\left[  |\nabla_{g_{0}}u|_{g_{0}}^{2}+bu^{2}\right]
dV_{g_{0}}}{\left(  \int_{\mathbb{S}^{m}}|u|^{2^{\ast}}dV_{g_{0}}\right)
^{2/2^{\ast}}}=\inf_{\substack{v\in D^{1,2}(\mathbb{R}^{m})\\v\neq0}%
}\frac{\int_{\mathbb{R}^{m}}\left[  \left\vert \nabla v\right\vert
^{2}dx+\widetilde{b}v^{2}\right]  dx}{\left(  \int_{\mathbb{R}^{m}}\left\vert
v\right\vert ^{2^{\ast}}dx\right)  ^{2/2^{\ast}}}=:S_{b}.
\]
If $b\equiv\frac{m(m-2)}{4}$ then $\widetilde{b}\equiv0$ and $S_{\frac
{m(m-2)}{4}}=:S$ is the best Sobolev constant for the embedding $D^{1,2}%
(\mathbb{R}^{m})\hookrightarrow L^{2^{\ast}}(\mathbb{R}^{m}).$ This constant
is attained at the \textit{standard bubble}
\[
U(x)=\left[  m(m-2)\right]  ^{\frac{m-2}{4}}\left(  \frac{1}{1+\left\vert
x\right\vert ^{2}}\right)  ^{\frac{m-2}{2}}%
\]
and at any dilation $U_{\varepsilon}(x):=\varepsilon^{\frac{2-m}{2}}U\left(
\frac{x}{\varepsilon}\right)  $ of it, with $\varepsilon>0.$

\begin{lemma}
If $b\geq\frac{m(m-2)}{4}$ then $S_{b}=S.$
\end{lemma}

\begin{proof}
Clearly, $S_{b}\geq S.$ Fix $\alpha\in(\frac{1}{2},1).$ Then, for all
$\varepsilon\in(0,1),$
\[
\widetilde{b}(x)U_{\varepsilon}^{2}(x)\leq C\left(  \frac{1}{1+\left\vert
x\right\vert ^{2}}\right)  ^{2}\left(  \frac{\varepsilon}{\varepsilon
^{2}+|x|^{2}}\right)  ^{m-2}\leq C\varepsilon^{m-2}\left(  \frac
{1}{\varepsilon^{2}+|x|^{2}}\right)  ^{m-2+\alpha}.
\]
Hence, we have that
\begin{align*}
0  &  \leq\int_{\mathbb{R}^{m}}\widetilde{b}(x)U_{\varepsilon}^{2}%
(x)dx=\int_{\left\vert x\right\vert \leq\varepsilon}\widetilde{b}%
(x)U_{\varepsilon}^{2}(x)dx+\int_{\left\vert x\right\vert \geq\varepsilon
}\widetilde{b}(x)U_{\varepsilon}^{2}(x)dx\\
&  \leq C\varepsilon^{2}\int_{\left\vert y\right\vert \leq1}U^{2}%
(y)dy+C\varepsilon^{m-2}\int_{\left\vert x\right\vert \geq\varepsilon
}|x|^{-2m+4-2\alpha}dx\\
&  =C\varepsilon^{2}+C\varepsilon^{2(1-\alpha)}\longrightarrow0\text{\qquad as
\ }\varepsilon\rightarrow0.
\end{align*}
Therefore,
\[
\lim_{\varepsilon\rightarrow0}\frac{\int_{\mathbb{R}^{m}}\left(  \left\vert
\nabla U_{\varepsilon}\right\vert ^{2}+\widetilde{b}U_{\varepsilon}%
^{2}\right)  dx}{\left(  \int_{\mathbb{R}^{m}}\left\vert U_{\varepsilon
}\right\vert ^{2^{\ast}}dx\right)  ^{2/2^{\ast}}}=\frac{\int_{\mathbb{R}^{m}%
}\left\vert \nabla U_{\varepsilon}\right\vert ^{2}dx}{\left(  \int
_{\mathbb{R}^{m}}\left\vert U_{\varepsilon}\right\vert ^{2^{\ast}}dx\right)
^{2/2^{\ast}}}=S.
\]
This shows that $S\geq S_{b}.$
\end{proof}

\begin{proof}
[\textbf{Proof of Proposition \ref{prop:nonexistence}}]If $S_{b}$ were
attained at some $v\in D^{1,2}(\mathbb{R}^{m})$ then, as $\widetilde{b}\geq0$
and $\widetilde{b}\not \equiv 0$, we would have that%
\[
S=S_{b}=\frac{\int_{\mathbb{R}^{m}}\left(  \left\vert \nabla v\right\vert
^{2}+\widetilde{b}v^{2}\right)  dx}{\left(  \int_{\mathbb{R}^{m}}\left\vert
v\right\vert ^{2^{\ast}}dx\right)  ^{2/2^{\ast}}}>\frac{\int_{\mathbb{R}^{m}%
}\left\vert \nabla v\right\vert ^{2}dx}{\left(  \int_{\mathbb{R}^{m}%
}\left\vert v\right\vert ^{2^{\ast}}dx\right)  ^{2/2^{\ast}}}\geq S.
\]
This is a contradiction.
\end{proof}

\section{Compactness}

\label{Section : Compactness}A classical result by Struwe \cite{s} provides a
complete description of the lack of compactness of the energy functional for
critical problems in a bounded smooth domain of $\mathbb{R}^{m}.$ Anisotropic
critical problems with symmetries were treated in \cite{cf2}. Palais-Smale
sequences of positive functions for some Yamabe-type problems on a closed
manifold were described by Druet, Hebey and Robert in \cite{dhr}, and
symmetric ones were treated in \cite{sa}.

In this section we apply concentration compactness methods to prove Theorem
\ref{Theorem : Compactness}.

Throughout this section, $(M,g)$ is a closed Riemannian manifold of dimension
$m\geq3$, $\Gamma$ is a closed subgroup of $\text{Isom}_{g}(M),$ and
$a,b,c\in\mathcal{C}^{\infty}(M)$ are $\Gamma$-invariant functions with
$a,c>0.$ We shall not assume that $-\,$div$_{g}(a\nabla_{g})+b$ is coercive,
except when we prove Theorem \ref{Theorem : Compactness}.

We use the notation introduced in the previous section. We start with the
following fact.

\begin{lemma}
\label{lem:PSbdd}Every Palais-Smale sequence for the functional $J_{g}$ is
bounded in $H_{g}^{1}(M).$
\end{lemma}

\begin{proof}
Hereafter, $C$ will denote a positive constant, not necessarily the same one.
Let $(u_{k})$ be a sequence in $H_{g}^{1}(M)$ such that $J_{g}(u_{k}%
)\rightarrow\tau$ and $J_{g}^{\prime}(u_{k})\rightarrow0$ in $\left(
H_{g}^{1}(M)\right)  ^{\prime}.$ Then,
\[
|u_{k}|_{g,2^{\ast}}^{2^{\ast}}\leq C\left(  \frac{1}{m}|u_{k}|_{g,c,2^{\ast}%
}^{2^{\ast}}\right)  =C\left(  J_{g}(u_{k})-\frac{1}{2}J_{g}^{\prime}%
(u_{k})u_{k}\right)  \leq C+o(\Vert u_{k}\Vert_{g}).
\]
Hence,%
\begin{equation}
\int_{M}\left[  a|\nabla_{g}u_{k}|_{g}^{2}+b|u_{k}|^{2}\right]  dV_{g}%
=2\left(  J_{g}(u_{k})+\frac{1}{2^{\ast}}|u_{k}|_{g,c,2^{\ast}}^{2^{\ast}%
}\right)  \leq C+o(\Vert u_{k}\Vert_{g}). \label{eq:PSbdd1}%
\end{equation}
Moreover, as $M$ is compact, using H\"{o}lder's inequality we obtain%
\begin{equation}
|u_{k}|_{g,2}^{2}\leq C|u_{k}|_{g,2^{\ast}}^{2}\leq C+o(\Vert u_{k}\Vert
_{g}^{2/2^{\ast}}). \label{eq:PSbdd2}%
\end{equation}
As $b$ is bounded, inequalities (\ref{eq:PSbdd1}) and (\ref{eq:PSbdd2}) yield%
\begin{align*}
a_{0}\Vert u_{k}\Vert_{g}^{2}  &  \leq\int_{M}\left[  a|\nabla_{g}u_{k}%
|_{g}^{2}+b|u_{k}|^{2}\right]  dV_{g}+\int_{M}\left(  -b+a_{0}\right)
u_{k}^{2}\,dV_{g}\\
&  \leq\int_{M}\left[  a|\nabla_{g}u_{k}|_{g}^{2}+b|u_{k}|^{2}\right]
dV_{g}+C|u_{k}|_{g,2}^{2}\\
&  \leq C+o(\Vert u_{k}\Vert_{g})+o(\Vert u_{k}\Vert_{g}^{2/2^{\ast}}),
\end{align*}
where $a_{0}:=\min_{M}a.$ This implies that $(u_{k})$ is bounded in $H_{g}%
^{1}(M).$
\end{proof}

Next, we consider the problem
\begin{equation}
\left\{
\begin{tabular}
[c]{c}%
$-\Delta v=|v|^{2^{\ast}-2}v,$\\
$v\in D^{1,2}(\mathbb{R}^{m}),$%
\end{tabular}
\right.  \label{Problem : limit}%
\end{equation}
and its associated energy functional
\[
J_{\infty}(v):=\frac{1}{2}\int_{\mathbb{R}^{m}}|\nabla v|^{2}dx-\frac
{1}{2^{\ast}}\int_{\mathbb{R}^{m}}|v|^{2^{\ast}}dx,\qquad v\in D^{1,2}%
(\mathbb{R}^{m}).
\]
The proof of Theorem \ref{Theorem : Compactness} will follow easily from the
following proposition.

\begin{proposition}
\label{Prop : Compactness}Assume that $b\equiv0.$ Let $(u_{k})\ $be a $\Gamma
$-invariant Palais-Smale sequence for $J_{g}$ at the level $\tau>0$ such that
$u_{k}\rightharpoonup0$ weakly in $H_{g}^{1}(M)$ but not strongly. Then, after
passing to a subsequence, there exist a point $p\in M$ and a nontrivial
solution $\widehat{v}$ to problem \eqref{Problem : limit} such that $\#\Gamma
p<\infty$ and
\begin{equation}
\tau\geq\left(  \frac{a(p)^{m/2}\,\#\Gamma p}{c(p)^{\left(  m-2\right)  /2}%
}\right)  J_{\infty}(\widehat{v})\geq\left(  \min_{q\in M}\frac{a(q)^{m/2}%
\,\#\Gamma q}{c(q)^{\left(  m-2\right)  /2}}\right)  \frac{1}{m}S^{m/2}.
\label{Prop : Compactess Inequality d}%
\end{equation}

\end{proposition}

\begin{proof}
Fix $\delta$ such that $3\delta\in(0,i_{g}),$ where $i_{g}$ is the injectivity
radius of $M.$ As $M$ is compact, there is a constant $C_{1}>1$ such that, for
every $q\in M,$ $\varrho\in(0,3\delta],$ $\varphi\in\mathcal{C}^{\infty}(M)$
and $s\in\lbrack1,\infty),$%
\begin{align}
C_{1}^{-1}\int_{B(0,\varrho)}\left\vert \widetilde{\varphi}\right\vert ^{s}dx
&  \leq\int_{B_{g}(q,\varrho)}\left\vert \varphi\right\vert ^{s}dV_{g}\leq
C_{1}\int_{B(0,\varrho)}\left\vert \widetilde{\varphi}\right\vert
^{s}dx,\label{eq:transformation1}\\
C_{1}^{-1}\int_{B(0,\varrho)}\left\vert \nabla\tilde{\varphi}\right\vert
^{2}dx  &  \leq\int_{B_{g}(q,\varrho)}\left\vert \nabla_{g}\varphi\right\vert
_{g}^{2}\,dV_{g}\leq C_{1}\int_{B(0,\varrho)}\left\vert \nabla\tilde{\varphi
}\right\vert ^{2}dx, \label{eq:transformation2}%
\end{align}
where $\tilde{\varphi}:=\varphi\circ\exp_{q}$ is written in normal coordinates
around $q$ and $\left\vert \cdot\right\vert $ is the standard Euclidean metric.

By Lemma \ref{lem:PSbdd} we have that
\[
|u_{k}|_{g,c,2^{\ast}}^{2^{\ast}}=m\left(  J_{g}(u_{k})-\frac{1}{2}%
J_{g}^{\prime}(u_{k})u_{k}\right)  \rightarrow m\tau=:\beta>0.
\]
So, since $M$ is compact, after passing to a subsequence, there exist
$q_{0}\in M$ and $\lambda_{0}\in(0,\beta)$ such that
\[
\int_{B_{g}(q_{0},\delta)}c|u_{k}|^{2^{\ast}}dV_{g}\geq\lambda_{0}%
\qquad\forall k\in\mathbb{N},
\]
where $B_{g}(q,r)$ denotes the ball in $(M,g)$ with center $q$ and radius $r.$
For each $k$, the Levy concentration function $Q_{k}:[0,\infty)\rightarrow
\lbrack0,\infty)$ given by
\[
Q_{k}(r):=\max_{q\in M}\int_{B_{g}(q,r)}c|u_{k}|^{2^{\ast}}dV_{g}%
\]
is continuous, nondecreasing, and satisfies $Q_{k}(0)=0$ and $Q_{k}%
(\delta)\geq\lambda_{0}.$ We fix $\lambda\in(0,\lambda_{0})$ such that%
\begin{equation}
\lambda<C_{1}^{-m-1}(\min_{M}c)\left[  \frac{1}{2}S(\min_{M}a)(\max_{M}%
c)^{-1}\right]  ^{m/2}. \label{eq:lambda}%
\end{equation}
Then, for each $k\in\mathbb{N}$, there exist $p_{k}\in M$ and $r_{k}%
\in(0,\delta]$ such that
\begin{equation}
Q_{k}(r_{k})=\int_{B_{g}(p_{k},r_{k})}c|u_{k}|^{2^{\ast}}dV_{g}=\lambda
\label{eq:concentration}%
\end{equation}
and, after passing to a subsequence, $p_{k}\rightarrow p$ in $M.$

Fix a cut-off function $\zeta\in\mathcal{C}_{c}^{\infty}(\mathbb{R}^{m})$ such
that $0\leq\zeta\leq1,$ $\zeta(y)=1$ if $\left\vert y\right\vert \leq2\delta$
and $\zeta(y)=0$ if $\left\vert y\right\vert \geq3\delta$ and, for each $k,$
define
\[
v_{k}(x):=r_{k}^{(m-2)/2}(u_{k}\circ\exp_{p_{k}})(r_{k}x),\text{\qquad}%
\zeta_{k}(x):=\zeta(r_{k}x),
\]%
\[
a_{k}(x):=(a\circ\exp_{p_{k}})(r_{k}x)\text{\qquad and\qquad}c_{k}%
(x):=(c\circ\exp_{p_{k}})(r_{k}x).
\]
Then, supp$(\zeta_{k}v_{k})\subset\overline{B(0,3\delta r_{k}^{-1})}$ and,
extending $\zeta_{k}v_{k}$ by $0$ outside $B(0,3\delta r_{k}^{-1}),$ we have
that $\zeta_{k}v_{k}\in\mathcal{C}_{c}^{\infty}(\mathbb{R}^{m})\subset
D^{1,2}(\mathbb{R}^{m}).$ As $\zeta\equiv1$ in $B(0,r_{k}),$ using
(\ref{eq:concentration}) and (\ref{eq:transformation1}) and performing the
change of variable $y=r_{k}x$ we obtain%
\begin{align}
0  &  <\lambda=\int_{B_{g}(p_{k},r_{k})}c|u_{k}|^{2^{\ast}}dV_{g}\leq
C_{1}\int_{B(0,r_{k})}(c\circ\exp_{p_{k}})|\zeta(u_{k}\circ\exp_{p_{k}%
})|^{2^{\ast}}dy\label{eq:nontriviality}\\
&  =C_{1}\int_{B(0,1)}c_{k}|\zeta_{k}v_{k}|^{2^{\ast}}dx\leq C\int
_{B(0,1)}|\zeta_{k}v_{k}|^{2^{\ast}}dx.\nonumber
\end{align}
Here and hereafter $C$ stands for a positive constant, not necessarily the
same one. Moreover, inequalities (\ref{eq:transformation1}) and
(\ref{eq:transformation2}) yield%
\begin{align*}
&  \int_{B(0,3\delta r_{k}^{-1})}\left\vert \nabla\left(  \zeta_{k}%
v_{k}\right)  \right\vert ^{2}\,dx=\int_{B(0,3\delta)}\left\vert \nabla
(\zeta(u_{k}\circ\exp_{p_{k}}))\right\vert ^{2}\,dy\\
&  \qquad\leq C\int_{B(0,3\delta)}\left[  \zeta^{2}\left\vert \nabla
(u_{k}\circ\exp_{p_{k}})\right\vert ^{2}+\left\vert \nabla\zeta\right\vert
^{2}(u_{k}\circ\exp_{p_{k}})^{2}\right]  \,dy\\
&  \qquad\leq C\int_{B(0,3\delta)}\left[  \left\vert \nabla(u_{k}\circ
\exp_{p_{k}})\right\vert ^{2}+(u_{k}\circ\exp_{p_{k}})^{2}\right]  \,dy\\
&  \qquad\leq C\int_{B_{g}(p_{k},3\delta)}\left[  \left\vert \nabla_{g}%
u_{k}\right\vert _{g}^{2}+u_{k}^{2}\right]  \,dV_{g},
\end{align*}
so Lemma \ref{lem:PSbdd} implies that $\left(  \zeta_{k}v_{k}\right)  $ is
bounded in $D^{1,2}(\mathbb{R}^{m}).$ Therefore, after passing to a
subsequence, we have that $\zeta_{k}v_{k}\rightharpoonup v$ weakly in
$D^{1,2}(\mathbb{R}^{m}),$ $\zeta_{k}v_{k}\rightarrow v$ in $L_{loc}%
^{2}(\mathbb{R}^{m})$ and $\zeta_{k}v_{k}\rightarrow v$ a.e. in $\mathbb{R}%
^{m}.$ The proof of the proposition will follow from the next three claims.

\textsc{Claim 1. }$v\neq0.$

To prove this claim first note that, as $M$ is compact, there exists $C_{2}>1$
such that, for every $q\in M,$
\begin{equation}
C_{2}^{-1}\left\vert y-z\right\vert \leq d_{g}(\exp_{q}\left(  y\right)
,\exp_{q}\left(  z\right)  )\leq C_{2}\left\vert y-z\right\vert \text{\qquad
}\forall y,z\in B(0,2\delta), \label{eq:transformation3}%
\end{equation}
where $d_{g}$ is the distance in $M.$ Set $\varrho:=C_{2}^{-1}.$ Then, for
every $z\in\overline{B(0,1)}$ we have that%
\[
\exp_{p_{k}}B(r_{k}z,\,r_{k}\varrho)\subset B_{g}(\exp_{p_{k}}(r_{k}%
z),\,r_{k}).
\]
Now, arguing by contradiction, assume that $v=0.$ Let $\vartheta\in
\mathcal{C}_{c}^{\infty}(\mathbb{R}^{m})$ be such that supp$(\vartheta)\subset
B(z,\varrho)$ for some $z\in\overline{B(0,1)}.$ Then, supp$(\vartheta)\subset
B(0,2).$ Set $\hat{\vartheta}_{k}(q):=\vartheta(r_{k}^{-1}\exp_{p_{k}}%
^{-1}(q)).$ As $\zeta_{k}\equiv1$ in $B(0,2),$ $\zeta_{k}v_{k}\rightarrow0$ in
$L_{loc}^{2}(\mathbb{R}^{m}),$ $J_{g}^{\prime}(u_{k})\rightarrow0$ in $\left(
H_{g}^{1}(M)\right)  ^{\prime}$ and $(\hat{\vartheta}_{k}^{2}u_{k})$ is
bounded in $H_{g}^{1}(M),$ using inequalities (\ref{eq:transformation1}) and
(\ref{eq:transformation2}) and H\"{o}lder's and Sobolev's inequalities, we
obtain%
\begin{align*}
&  \int_{\mathbb{R}^{m}}\left\vert \nabla\left(  \vartheta\zeta_{k}%
v_{k}\right)  \right\vert ^{2}dx=\int_{B(0,2)}\left\vert \nabla\left(
\vartheta v_{k}\right)  \right\vert ^{2}dx=\int_{B(0,2r_{k})}\left\vert
\nabla\left(  (\hat{\vartheta}_{k}u_{k})\circ\exp_{p_{k}}\right)  \right\vert
^{2}dy\\
&  \leq C_{3}\int_{B_{g}(p_{k},2r_{k})}a\left\vert \nabla_{g}(\hat{\vartheta
}_{k}u_{k})\right\vert _{g}^{2}dV_{g}\\
&  =C_{3}\int_{B_{g}(p_{k},2r_{k})}a\left[  \hat{\vartheta}_{k}^{2}\left\vert
\nabla_{g}\left(  u_{k}\right)  \right\vert _{g}^{2}+2\hat{\vartheta}_{k}%
u_{k}\left\langle \nabla_{g}u_{k},\nabla_{g}\hat{\vartheta}_{k}\right\rangle
_{g}+\left\vert \nabla_{g}\hat{\vartheta}_{k}\right\vert _{g}^{2}u_{k}%
^{2}\right]  dV_{g}\\
&  =C_{3}\int_{B_{g}(p_{k},2r_{k})}a\left\langle \nabla_{g}u_{k},\nabla
_{g}(\hat{\vartheta}_{k}^{2}u_{k})\right\rangle _{g}\,dV_{g}+o(1)\\
&  =C_{3}\int_{B_{g}(p_{k},2r_{k})}c\left\vert u_{k}\right\vert ^{2^{\ast}%
-2}(\hat{\vartheta}_{k}u_{k})^{2}\,dV_{g}+o(1)\\
&  \leq C_{4}\int_{B(0,2)\cap B(z,\rho)}\left\vert v_{k}\right\vert ^{2^{\ast
}-2}(\vartheta v_{k})^{2}\,dx+o(1)\\
&  \leq C_{4}\left(  \int_{B(z,\rho)}\left\vert v_{k}\right\vert ^{2^{\ast}%
}dx\right)  ^{2/m}\left(  \int_{B(0,2)}\left\vert \vartheta\zeta_{k}%
v_{k}\right\vert ^{2^{\ast}}dx\right)  ^{2/2^{\ast}}+o(1)\\
&  \leq C_{4}S^{-1}\left(  \int_{B(z,\rho)}\left\vert v_{k}\right\vert
^{2^{\ast}}dx\right)  ^{2/m}\int_{\mathbb{R}^{m}}\left\vert \nabla\left(
\vartheta\zeta_{k}v_{k}\right)  \right\vert ^{2}\,dx+o(1),
\end{align*}
where $C_{3}:=C_{1}(\min_{M}a)^{-1}$ and $C_{4}:=C_{1}(\max_{M}c)C_{3}.$ On
the other hand, from (\ref{eq:transformation1}), (\ref{eq:transformation3})
and (\ref{eq:concentration}) we derive%
\begin{align*}
\int_{B(z,\rho)}\left\vert v_{k}\right\vert ^{2^{\ast}}dx  &  \leq C_{1}%
(\min_{M}c)^{-1}\int_{B_{g}(\exp_{p_{k}}(r_{k}z),\,r_{k})}c\left\vert
u_{k}\right\vert ^{2^{\ast}}\,dV_{g}\\
&  \leq C_{1}(\min_{M}c)^{-1}\lambda.
\end{align*}
It follows from (\ref{eq:lambda}) that $(C_{1}(\min_{M}c)^{-1}\lambda
)^{2/m}<\frac{1}{2}C_{4}^{-1}S.$ Therefore,%
\[
\lim_{k\rightarrow\infty}\int_{\mathbb{R}^{m}}\left\vert \nabla\left(
\vartheta\zeta_{k}v_{k}\right)  \right\vert ^{2}\,dx=0
\]
and Sobolev's inequality yields
\[
\lim_{k\rightarrow\infty}\int_{\mathbb{R}^{m}}\left\vert \vartheta\zeta
_{k}v_{k}\right\vert ^{2^{\ast}}dx=0
\]
for every $\vartheta\in\mathcal{C}_{c}^{\infty}(\mathbb{R}^{m})$ such that
supp$(\vartheta)\subset B(z,\varrho)$ for some $z\in\overline{B(0,1)}$. As
$B(0,1)$ can be covered by a finite number of balls $B(z_{j},\varrho)$ with
$z_{j}\in\overline{B(0,1)},$ choosing a partition of unity $\{\vartheta
_{j}^{2^{\ast}}\}$ subordinated to this covering, we conclude that%
\[
\int_{B(0,1)}\left\vert \zeta_{k}v_{k}\right\vert ^{2^{\ast}}dx\leq%
{\textstyle\sum\limits_{j}}
\int_{\mathbb{R}^{m}}\left\vert \vartheta_{j}\zeta_{k}v_{k}\right\vert
^{2^{\ast}}dx\longrightarrow0,
\]
contradicting (\ref{eq:nontriviality}). This finishes the proof of Claim 1.

\textsc{Claim 2. }$\widehat{v}:=\left(  \frac{c(p)}{a(p)}\right)  ^{(m-2)/4}v$
\ is a nontrivial solution to problem (\ref{Problem : limit}).

First we show that, after passing to a subsequence, $r_{k}\rightarrow0$.
Arguing by contradiction, assume that $r_{k}>\theta>0$ for all $k$ large
enough. Then, as $\zeta_{k}v_{k}\rightarrow v$ a.e. in $\mathbb{R}^{m}$,
supp$(\zeta_{k}v_{k})\subset\overline{B(0,3\delta r_{k}^{-1})}$, $v\neq0$ and
$\zeta_{k}v_{k}\rightarrow v$ in $L_{loc}^{2}(\mathbb{R}^{m}),$ using
inequality (\ref{eq:transformation1}) we obtain%
\begin{align*}
0  &  \neq\int_{B(0,3\delta\theta^{-1})}\left\vert v\right\vert ^{2}%
dx=\int_{B(0,3\delta\theta^{-1})}\left\vert \zeta_{k}v_{k}\right\vert
^{2}dx+o(1)\\
&  =r_{k}^{-2}\int_{B(0,3\delta)}\left\vert \zeta(u_{k}\circ\exp_{p_{k}%
})\right\vert ^{2}dy+o(1)\\
&  \leq C_{1}\theta^{-2}\int_{M}|u_{k}|^{2}dV_{g}.
\end{align*}
This yields a contradiction because, as we are assuming that $u_{k}%
\rightharpoonup0$ weakly in $H_{g}^{1}(M)$, we have that $u_{k}\rightarrow0$
strongly in $L_{g}^{2}(M).$

Claim 2 is equivalent to showing that $v$ satisfies%
\[
-a(p)\Delta v=c(p)|v|^{2^{\ast}-2}v,\qquad v\in D^{1,2}(\mathbb{R}^{m}),
\]
i.e. we need to show that
\begin{equation}
\int_{\mathbb{R}^{m}}a(p)\left\langle \nabla v,\nabla\varphi\right\rangle
\,dx=\int_{\mathbb{R}^{m}}c(p)|v|^{2^{\ast}-2}v\varphi\,dx\qquad\forall
\varphi\in C_{c}^{\infty}(\mathbb{R}^{m}). \label{eq:claim2}%
\end{equation}
To this end, take $\varphi\in C_{c}^{\infty}(\mathbb{R}^{m})$ and let $R>0$ be
such that supp$(\varphi)\subset B(0,R)$. For $k$ such that $Rr_{k}<3\delta$
define $\hat{\varphi}_{k}\in H_{g}^{1}(M)$ by
\[
\hat{\varphi}_{k}(q):=r_{k}^{\frac{2-m}{2}}\varphi(r_{k}^{-1}\exp_{p_{k}}%
^{-1}(q)).
\]
Note first that, as $a_{k}\rightarrow a(p)$ and $c_{k}\rightarrow c(p)$ in
$L_{loc}^{\infty}(\mathbb{R}^{m})$ and $\zeta_{k}v_{k}\rightharpoonup v$
weakly in $D^{1,2}(\mathbb{R}^{m})$ we have that
\begin{align*}
\int_{\mathbb{R}^{m}}a_{k}\left\langle \nabla\left(  \zeta_{k}v_{k}\right)
,\nabla\varphi\right\rangle \,dx  &  =\int_{\mathbb{R}^{m}}a(p)\left\langle
\nabla v,\nabla\varphi\right\rangle \,dx+o(1),\\
\int_{\mathbb{R}^{m}}c_{k}|\zeta_{k}v_{k}|^{2^{\ast}-2}\left(  \zeta_{k}%
v_{k}\right)  \varphi\,dx  &  =\int_{\mathbb{R}^{m}}c(p)|v|^{2^{\ast}%
-2}v\varphi\,dx+o(1).
\end{align*}
Next observe that, if $(g_{ij}^{k})$ is the metric $g$ written in normal
coordinates around $p_{k}$, $(g_{k}^{ji})$ is its inverse, $\left\vert
g^{k}\right\vert :=\det(g_{ij}^{k})$ and $(\partial^{ji})$ is the identity
matrix then, for every $i,j=1,...,m,$%
\begin{equation}
\lim_{\left\vert y\right\vert \rightarrow0}g_{k}^{ji}(y)=\partial
^{ji}\text{\quad and\quad}\lim_{\left\vert y\right\vert \rightarrow
0}\left\vert g^{k}\right\vert ^{1/2}(y)=1, \label{eq:normal}%
\end{equation}
uniformly in $k$. Therefore, as supp$(\hat{\varphi}_{k}\circ\exp_{p_{k}%
})\subset B(0,Rr_{k}),$ $r_{k}\rightarrow0,$ and $(u_{k}\circ\exp_{p_{k}})$
and $(\hat{\varphi}_{k}\circ\exp_{p_{k}})$ are bounded in $D^{1,2}%
(\mathbb{R}^{m}),$ we have that
\begin{align*}
&  \int_{\mathbb{R}^{m}}(a\circ\exp_{p_{k}})\left\langle \nabla(u_{k}\circ
\exp_{p_{k}}),\nabla(\hat{\varphi}_{k}\circ\exp_{p_{k}})\right\rangle
dy-\int_{M}a\left\langle \nabla_{g}u_{k},\nabla_{g}\hat{\varphi}%
_{k}\right\rangle _{g}dV_{g}\\
&  =\sum_{i,j}\int_{B(0,Rr_{k})}(a\circ\exp_{p_{k}})(\partial^{ji}-\left\vert
g^{k}\right\vert ^{1/2}g_{k}^{ji})\,\partial_{i}(u_{k}\circ\exp_{p_{k}%
})\,\partial_{j}(\hat{\varphi}_{k}\circ\exp_{p_{k}})\,dy\\
&  =o(1),
\end{align*}
and%
\begin{align*}
&  \int_{\mathbb{R}^{m}}(c\circ\exp_{p_{k}})|u_{k}\circ\exp_{p_{k}}|^{2^{\ast
}-2}(u_{k}\circ\exp_{p_{k}})(\hat{\varphi}_{k}\circ\exp_{p_{k}})dy-\int
_{M}c\left\vert u_{k}\right\vert ^{2^{\ast}-2}u_{k}\hat{\varphi}_{k}dV_{g}\\
&  =\int_{B(0,Rr_{k})}(c\circ\exp_{p_{k}})|u_{k}\circ\exp_{p_{k}}|^{2^{\ast
}-2}(u_{k}\circ\exp_{p_{k}})(\hat{\varphi}_{k}\circ\exp_{p_{k}})(1-\left\vert
g^{k}\right\vert ^{1/2})\,dy\\
&  =o(1).
\end{align*}
Finally, as $J_{g}^{\prime}(u_{k})\rightarrow0$ in $\left(  H_{g}%
^{1}(M)\right)  ^{\prime}$ and $(\hat{\varphi}_{k})$ is bounded in $H_{g}%
^{1}(M)$ we conclude that, for $k$ large enough,%
\begin{align*}
&  \int_{\mathbb{R}^{m}}a(p)\left\langle \nabla v,\nabla\varphi\right\rangle
\,dx\\
&  \qquad=\int_{\mathbb{R}^{m}}a_{k}\left\langle \nabla\left(  \zeta_{k}%
v_{k}\right)  ,\nabla\varphi\right\rangle \,dx+o(1)\\
&  \qquad=\int_{\mathbb{R}^{m}}(a\circ\exp_{p_{k}})\left\langle \nabla
(u_{k}\circ\exp_{p_{k}}),\nabla(\hat{\varphi}_{k}\circ\exp_{p_{k}%
})\right\rangle \,dy+o(1)\\
&  \qquad=\int_{M}a\left\langle \nabla_{g}u_{k},\nabla_{g}\hat{\varphi}%
_{k}\right\rangle _{g}dV_{g}+o(1)\\
&  \qquad=\int_{M}c\left\vert u_{k}\right\vert ^{2^{\ast}-2}u_{k}\hat{\varphi
}_{k}\,dV_{g}+o(1)\\
&  \qquad=\int_{\mathbb{R}^{m}}(c\circ\exp_{p_{k}})|u_{k}\circ\exp_{p_{k}%
}|^{2^{\ast}-2}(u_{k}\circ\exp_{p_{k}})(\hat{\varphi}_{k}\circ\exp_{p_{k}%
})\,dy+o(1)\\
&  \qquad=\int_{\mathbb{R}^{m}}c_{k}|\zeta_{k}v_{k}|^{2^{\ast}-2}\left(
\zeta_{k}v_{k}\right)  \varphi\,dx+o(1)\\
&  \qquad=\int_{\mathbb{R}^{m}}c(p)|v|^{2^{\ast}-2}v\varphi\,dx+o(1).
\end{align*}
This proves (\ref{eq:claim2}).

\textsc{Claim 3. }$\#\Gamma p<\infty$ and $\tau\geq\left(  \frac
{a(p)^{m/2}\,\#\Gamma p}{c(p)^{\left(  m-2\right)  /2}}\right)  J_{\infty
}(\widehat{v}).$

Let $\gamma_{1}p,...,\gamma_{n}p$ be $n$ distinct points in the $\Gamma$-orbit
$\Gamma p$ of $p,$ and fix $\eta\in(0,\delta]$ such that $d_{g}(\gamma
_{i}p,\gamma_{j}p)\geq4\eta$ if $i\neq j.$ For $k$ sufficiently large,
$d_{g}(p_{k},p)<\eta$ so, as $\gamma_{i}$ is an isometry, we have that
$d_{g}(\gamma_{i}p_{k},\gamma_{j}p_{k})>2\eta$ for all $k\in\mathbb{N}$ and
$i\neq j.$ Since $c$ and $u_{k}$ are $\Gamma$-invariant, for each $\rho
\in(0,\eta]$ we obtain that
\begin{equation}
n\int_{B_{g}(p_{k},\rho)}c\left\vert u_{k}\right\vert ^{2^{\ast}}dV_{g}%
=\sum_{i=1}^{n}\int_{B_{g}(\gamma_{i}p_{k},\rho)}c\left\vert u_{k}\right\vert
^{2^{\ast}}dV_{g}\leq\int_{M}c\left\vert u_{k}\right\vert ^{2^{\ast}}dV_{g}.
\label{eq:disjoint}%
\end{equation}
Let $\varepsilon>0.$ By (\ref{eq:normal}) there exists $\rho\in(0,\eta]$ such
that $(1+\varepsilon)^{-1}<\left\vert g^{k}\right\vert ^{1/2}<(1+\varepsilon)$
in $B(0,\rho)$ for $k$ large enough$.$ As $1_{B(0,\rho r_{k}^{-1})}%
c_{k}\rightarrow c(p)$ and $\zeta_{k}v_{k}\rightarrow v$ a.e. in
$\mathbb{R}^{m}$, Fatou's lemma and inequality (\ref{eq:disjoint}) yield%
\begin{align*}
\frac{n}{m}\int_{\mathbb{R}^{m}}c(p)\left\vert v\right\vert ^{2^{\ast}}dx  &
\leq\liminf_{k\rightarrow\infty}\frac{n}{m}\int_{B(0,\rho r_{k}^{-1})}%
c_{k}\left\vert \zeta_{k}v_{k}\right\vert ^{2^{\ast}}dx\\
&  \leq\liminf_{k\rightarrow\infty}\frac{n}{m}\int_{B(0,\rho)}(c\circ
\exp_{p_{k}})\left\vert u_{k}\circ\exp_{p_{k}}\right\vert ^{2^{\ast}}dy\\
&  \leq(1+\varepsilon)\liminf_{k\rightarrow\infty}\frac{n}{m}\int_{B_{g}%
(p_{k},\rho)}c\left\vert u_{k}\right\vert ^{2^{\ast}}dV_{g}\\
&  \leq(1+\varepsilon)\lim_{k\rightarrow\infty}\frac{1}{m}\int_{M}c\left\vert
u_{k}\right\vert ^{2^{\ast}}dV_{g}=(1+\varepsilon)\tau.
\end{align*}
This implies that $n$ is bounded and, therefore, $\#\Gamma p<\infty.$
Moreover, as $\varepsilon$ is arbitrary, taking $n=\#\Gamma p$, we conclude
that%
\begin{align*}
\left(  \frac{a(p)^{m/2}\,\#\Gamma p}{c(p)^{\left(  m-2\right)  /2}}\right)
J_{\infty}(\widehat{v})  &  =\left(  \frac{a(p)^{m/2}\,\#\Gamma p}%
{c(p)^{\left(  m-2\right)  /2}}\right)  \frac{1}{m}\int_{\mathbb{R}^{m}%
}\left\vert \widehat{v}\right\vert ^{2^{\ast}}dx\\
&  =\frac{\#\Gamma p}{m}\int_{\mathbb{R}^{m}}c(p)\left\vert v\right\vert
^{2^{\ast}}dx\leq\tau,
\end{align*}
as claimed.

This finishes the proof of the proposition.
\end{proof}

\begin{proof}
[Proof of Theorem \ref{Theorem : Compactness}]Let $(u_{k})$ be a sequence in
$H_{g}^{1}(M)^{\Gamma}$ such that $J_{g}(u_{k})\rightarrow\tau<(\min_{q\in
M}\frac{a(q)^{m/2}\,\#\Gamma q}{c(q)^{(m-2)/2}})\frac{1}{m}S^{m/2}$ and
$J_{g}^{\prime}(u_{k})\rightarrow0$ in $\left(  H_{g}^{1}(M)\right)  ^{\prime
}.$ By Lemma \ref{lem:PSbdd},\ $(u_{k})$ is bounded in $H_{g}^{1}(M)$ so,
after passing to a subsequence, $u_{k}\rightharpoonup u$ weakly in $H_{g}%
^{1}(M).$ It follows that $u\in H_{g}^{1}(M)^{\Gamma},$ $J_{g}^{\prime}(u)=0$
and, as $-\,$div$_{g}(a\nabla_{g})+b$ is coercive on $H_{g}^{1}(M)^{\Gamma},$
\begin{equation}
J_{g}(u)=\frac{1}{m}\Vert u\Vert_{g,a,b}^{2}\leq\liminf_{k\rightarrow\infty
}\frac{1}{m}\Vert u_{k}\Vert_{g,a,b}^{2}=\lim_{k\rightarrow\infty}J_{g}%
(u_{k})=\tau. \label{eq:weakly}%
\end{equation}
Set $\widetilde{u}_{k}:=u_{k}-u.$ Then $\widetilde{u}_{k}\rightharpoonup0$
weakly in $H_{g}^{1}(M)$ and, by a standard argument (see, e.g., \cite{cf2,
w}), $(\widetilde{u}_{k})$ is a $\Gamma$-invariant Palais-Smale sequence for
the functional $J_{g}$ with $b=0$ at the level $\widetilde{\tau}:=\tau
-J_{g}(u)<(\min_{q\in M}\frac{a(q)^{m/2}\,\#\Gamma q}{c(q)^{(m-2)/2}})\frac
{1}{m}S^{m/2}.$ Proposition \ref{Prop : Compactness}\ implies that
$\widetilde{\tau}=0.$ Thus, inequality (\ref{eq:weakly}) is an equality. It
follows that $u_{k}\rightarrow u$ strongly in $H_{g}^{1}(M).$
\end{proof}

\section{A variational principle for nodal solutions}

\label{sec:vp}This section is devoted to the proof of Theorem
\ref{thm: Variational principle}.

We begin by showing that a neighborhood of the set of functions in
$H^{1}_g(M)^{\Gamma}$ which do not change sign is invariant under the negative
gradient flow of $J_{g},$ with respect to a suitably chosen scalar product in
$H_g^{1}(M)^{\Gamma}$.

Since we are assuming that $a>0$ and the operator $-$div$_{g}(a\nabla_{g})+b$
is coercive on $H_g^{1}(M)^{\Gamma},$ there exists $\mu>0$ such that%
\begin{equation}
\int_{M}\left[  a|\nabla_{g}u|_{g}^{2}+b|u|^{2}\right]  dV_{g}\geq\mu\int
_{M}\left[  a|\nabla_{g}u|_{g}^{2}+|u|^{2}\right]  dV_{g}\text{\qquad}\forall
u\in H_{g}^{1}(M)^{\Gamma}. \label{eq:coercivity}%
\end{equation}
Fix $A>\max\{1,\mu,|b|_{\mathcal{C}^{0}(M)}\}$ and consider the scalar
product
\begin{equation}
\left\langle u,v\right\rangle _{g,a,A}:=\int_{M}\left[  a\langle\nabla
_{g}u,\nabla_{g}v\rangle_{g}+Auv\right]  dV_{g} \label{eq:scalar_product}%
\end{equation}
in $H_{g}^{1}(M)^{\Gamma}$. We write $\left\Vert \cdot\right\Vert _{g,a,A}$
for the induced norm, which is equivalent to the standard norm in $H_{g}%
^{1}(M)^{\Gamma}$. Given a subset $\mathcal{D}$ of $H_{g}^{1}(M)^{\Gamma}$ and
$\rho>0,$ we set%
\[
B_{\rho}(\mathcal{D}):=\{u\in H_{g}^{1}(M)^{\Gamma}:\text{dist}_{A}%
(u,\mathcal{D})\leq\rho\},
\]
where dist$_{A}(u,\mathcal{D}):=\inf_{v\in\mathcal{D}}\left\Vert
u-v\right\Vert _{g,a,A}.$

The gradient of the functional $J_{g}:H_{g}^{1}(M)^{\Gamma}\rightarrow
\mathbb{R}$ at $u\in H_{g}^{1}(M)^{\Gamma},$ with respect to the scalar
product (\ref{eq:scalar_product}), is the vector $\nabla J_{g}(u)$ which
satisfies%
\begin{align*}
&  \left\langle \nabla J_{g}(u),v\right\rangle _{g,a,A}=J_{g}^{\prime}(u)v\\
&  \qquad=\left\langle u,v\right\rangle _{g,a,A}-\int_{M}(A-b)uv\,dV_{g}%
-\int_{M}c\left\vert u\right\vert ^{2^{\ast}-2}uv\,dV_{g}\text{\qquad}\forall
v\in H_{g}^{1}(M)^{\Gamma},
\end{align*}
i.e., $\nabla J_{g}(u)=u-Lu-Gu$ where $Lu,$ $Gu\in H_{g}^{1}(M)^{\Gamma}$ are
the unique solutions to%
\begin{align}
-\text{div}_{g}(a\nabla_{g}(Lu))+A\left(  Lu\right)   &  =(A-b)u,
\label{eq:Lu}\\
-\text{div}_{g}(a\nabla_{g}(Gu))+A\left(  Gu\right)   &  =c\left\vert
u\right\vert ^{2^{\ast}-2}u. \label{eq:Gu}%
\end{align}
Then, the following inequality holds true. Its proof was suggested by
J\'{e}r\^{o}me V\'{e}tois and fills in a small gap in his proof of Lemma 2.1
in \cite{v}.

\begin{lemma}
\label{lem:vetois}Set $\overline{\mu}:=\frac{A-\mu}{A+\mu}\in(0,1).$ Then, for
every $u\in H_{g}^{1}(M)^{\Gamma},$ we have%
\[
\left\Vert Lu\right\Vert _{g,a,A}\leq\overline{\mu}\left\Vert u\right\Vert
_{g,a,A}.
\]

\end{lemma}

\begin{proof}
By (\ref{eq:coercivity}), for every $u\in H_{g}^{1}(M)^{\Gamma}$ we have that%
\begin{align*}
\int_{M}(A-b)u^{2}dV_{g}  &  \leq\int_{M}Au^{2}dV_{g}-\mu\int_{M}\left[
a|\nabla_{g}u|_{g}^{2}+|u|^{2}\right]  dV_{g}+\int_{M}a\left\vert \nabla
_{g}u\right\vert _{g}^{2}dV_{g}\\
&  \leq\frac{A-\mu}{A}\int_{M}\left[  a|\nabla_{g}u|_{g}^{2}+A|u|^{2}\right]
dV_{g}=\frac{A-\mu}{A}\left\Vert u\right\Vert _{g,a,A}^{2}.
\end{align*}
Hence, using (\ref{eq:Lu}) we obtain
\begin{align*}
\left\Vert Lu\right\Vert _{g,a,A}^{2}  &  =\int_{M}(A-b)u(Lu)\,dV_{g}\leq
\frac{1}{2}\int_{M}(A-b)\left[  u^{2}+(Lu)^{2}\right]  \,dV_{g}\\
&  \leq\frac{A-\mu}{2A}\left(  \left\Vert u\right\Vert _{g,a,A}^{2}+\left\Vert
Lu\right\Vert _{g,a,A}^{2}\right)  .
\end{align*}
Consequently,%
\[
\frac{A+\mu}{2A}\left\Vert Lu\right\Vert _{g,a,A}^{2}\leq\frac{A-\mu}%
{2A}\left\Vert u\right\Vert _{g,a,A}^{2},
\]
as claimed.
\end{proof}

We consider the negative gradient flow $\psi:\mathcal{G}\rightarrow H_{g}%
^{1}(M)^{\Gamma}$ of $J_{g},$ defined by
\[
\frac{\partial}{\partial t}\psi(t,u)=-\nabla J_{g}(\psi(t,u)),\qquad
\psi(0,u)=u,
\]
where $\mathcal{G}:=\{(t,u):u\in H_{g}^{1}(M)^{\Gamma},$ $0\leq t<T(u)\}$ and
$T(u)$ is the maximal existence time for the trajectory $t\mapsto\psi(t,u).$ A
subset $\mathcal{D}$ of $H_{g}^{1}(M)^{\Gamma}$ is said to be strictly
positively invariant if
\[
\psi(t,u)\in\text{int}\mathcal{D}\text{\qquad for every }u\in\mathcal{D}\text{
and }t\in(0,T(u)).
\]

The set of functions in $H_g^{1}(M)^{\Gamma}$ which do not change sign is
$\mathcal{P}^{\Gamma}\cup-\mathcal{P}^{\Gamma}$, where $\mathcal{P}^{\Gamma
}:=\{u\in H_{g}^{1}(M)^{\Gamma}:u\geq0\}$ is the convex cone of nonnegative
functions. The nodal solutions to the problem
(\ref{Equation : generalized Yamabe PDE}) lie in the set%
\[
\mathcal{E}_{g}^{\Gamma}:=\{u\in\mathcal{N}_{g}^{\Gamma}:u^{+},u^{-}%
\in\mathcal{N}_{g}^{\Gamma}\},
\]
where $u^{+}:=\max\{0,u\},$ $u^{-}:=\min\{0,u\}$ and $\mathcal{N}_{g}^{\Gamma
}$ is the Nehari manifold defined in (\ref{eq:nehari}).

\begin{lemma}
\label{lem:variational_principle}There exists $\rho_{0}>0$ such that, for
every $\rho\in(0,\rho_{0}),$

\begin{enumerate}
\item[(a)] $\left[  B_{\rho}(\mathcal{P}^{\Gamma})\cup B_{\rho}(-\mathcal{P}%
^{\Gamma})\right]  \cap\mathcal{E}_{g}^{\Gamma}=\emptyset$, and

\item[(b)] $B_{\rho}(\mathcal{P}^{\Gamma})$ and $B_{\rho}(-\mathcal{P}%
^{\Gamma})$ are strictly positively invariant.
\end{enumerate}
\end{lemma}

\begin{proof}
By symmetry considerations, it is enough to prove this for $B_{\rho
}(\mathcal{P}^{\Gamma}).$

(a): \ Note that $\left\vert u^{-}(p)\right\vert \leq\left\vert
u(p)-v(p)\right\vert $ for every $u,v:M\rightarrow\mathbb{R}$ with $v\geq0,$
$p\in M.$ Sobolev's inequality yields a positive constant $C$ such that%
\begin{equation}
\left\vert u^{-}\right\vert _{g,c,2^{\ast}}=\min_{v\in\mathcal{P}^{\Gamma}%
}\left\vert u-v\right\vert _{g,c,2^{\ast}}\leq C\min_{v\in\mathcal{P}^{\Gamma
}}\left\Vert u-v\right\Vert _{g,a,A}=C\,\text{dist}_{A}(u,\mathcal{P}^{\Gamma
}) \label{eq:distP}%
\end{equation}
for every $u\in H_{g}^{1}(M)^{\Gamma}.$ If $u\in\mathcal{E}_{g}^{\Gamma}$,
then $u^{-}\in\mathcal{N}_{g}^{\Gamma}$ and, therefore, $\left\vert
u^{-}\right\vert _{g,c,2^{\ast}}^{2^{\ast}}=mJ_{g}(u^{-})\geq m\tau
_{g}^{\Gamma}>0.$ This proves that dist$_{A}(u,\mathcal{P}^{\Gamma})\geq
\rho_{1}>0$ for all $u\in\mathcal{E}_{g}^{\Gamma}$.

(b): \ By the maximum principle, $Lv\in\mathcal{P}^{\Gamma}$ and
$Gv\in\mathcal{P}^{\Gamma}$ if $v\in\mathcal{P}^{\Gamma}.$ For $u\in H_{g}%
^{1}(M)^{\Gamma}$ let $v\in\mathcal{P}^{\Gamma}$ be such that dist$_{A}%
(u,\mathcal{P}^{\Gamma})=\left\Vert u-v\right\Vert _{g,a,A}.$ Then, Lemma
\ref{lem:vetois}\ yields%
\begin{equation}
\text{dist}_{A}(Lu,\mathcal{P}^{\Gamma})\leq\left\Vert Lu-Lv\right\Vert
_{g,a,A}\leq\overline{\mu}\left\Vert u-v\right\Vert _{g,a,A}=\overline{\mu
}\,\text{dist}_{A}(u,\mathcal{P}^{\Gamma}). \label{eq:distLu}%
\end{equation}
On the other hand, from (\ref{eq:Gu}), H\"{o}lder's inequality and
(\ref{eq:distP}) we get that%
\begin{align*}
&  \text{dist}_{A}(Gu,\mathcal{P}^{\Gamma})\left\Vert G(u)^{-}\right\Vert
_{g,a,A}\leq\left\Vert G(u)^{-}\right\Vert _{g,a,A}^{2}=\left\langle
G(u),G(u)^{-}\right\rangle _{g,a,A}\\
&  \text{\qquad}=\int_{M}c\left\vert u\right\vert ^{2^{\ast}-2}uG(u)^{-}%
\,dV_{g}\leq\int_{M}c\left\vert u^{-}\right\vert ^{2^{\ast}-2}u^{-}%
G(u)^{-}\,dV_{g}\\
&  \text{\qquad}\leq\left\vert u^{-}\right\vert _{g,c,2^{\ast}}^{2^{\ast}%
-1}\left\vert G(u)^{-}\right\vert _{g,c,2^{\ast}}\leq C^{2^{\ast}%
}\,\text{dist}_{A}(u,\mathcal{P}^{\Gamma})^{2^{\ast}-1}\left\Vert
G(u)^{-}\right\Vert _{g,a,A}.
\end{align*}
Hence,%
\begin{equation}
\text{dist}_{A}(Gu,\mathcal{P}^{\Gamma})\leq C^{2^{\ast}}\,\text{dist}%
_{A}(u,\mathcal{P}^{\Gamma})^{2^{\ast}-1}\text{\qquad}\forall u\in H_{g}%
^{1}(M)^{\Gamma}. \label{eq:distGu}%
\end{equation}
Fix $\nu\in\left(  \overline{\mu},1\right)  $ and let $\rho_{2}>0$ be such
that $C^{2^{\ast}}\rho_{2}^{2^{\ast}-2}\leq\nu-\overline{\mu}$. Then, for
$\rho\in(0,\rho_{2}),$ from inequalities (\ref{eq:distLu}) and
(\ref{eq:distGu}) we obtain%
\[
\text{dist}_{A}(Lu+Gu,\mathcal{P}^{\Gamma})\leq\nu\,\text{dist}_{A}%
(u,\mathcal{P}^{\Gamma})\text{\qquad}\forall u\in B_{\rho}(\mathcal{P}%
^{\Gamma}),\text{ }%
\]
Therefore, $Lu+Gu\in\,$int$B_{\rho}(\mathcal{P}^{\Gamma})$ if $u\in B_{\rho
}(\mathcal{P}^{\Gamma}).$ Since $B_{\rho}(\mathcal{P}^{\Gamma})$ is closed and
convex, Theorem 5.2 in \cite{d}\ yields that%
\[
\psi(t,u)\in B_{\rho}(\mathcal{P}^{\Gamma})\text{ \ for all \ }t\in
(0,T(u))\text{ \ if }u\in B_{\rho}(\mathcal{P}^{\Gamma}).
\]
Now we can argue as in the proof of Lemma 2 in \cite{cw}\ to show that, in
fact, $B_{\rho}(\mathcal{P}^{\Gamma})$ is strictly positively invariant. Letting $\rho_0:=\min\{\rho_1,\rho_2\}$,
we get the result.
\end{proof}

We fix $\rho\in(0,\rho_{0})$ and, for $d\in\mathbb{R}$, we set%
\[
\mathcal{D}_{d}^{\Gamma}:=B_{\rho}(\mathcal{P}^{\Gamma})\cup B_{\rho
}(\mathcal{P}^{\Gamma})\cup J_{g}^{d},
\]
where $J_{g}^{d}:=\{u\in H_{g}^{1}(M)^{\Gamma}:J_{g}(u)\leq d\}.$ It follows
from Lemma \ref{lem:variational_principle} that $\mathcal{D}_{0}^{\Gamma}$ is
strictly positively invariant under the flow $\psi$, and that a critical point
of $J_{g}$ is sign-changing iff it lies in the complement of $\mathcal{D}%
_{0}^{\Gamma}$.

To find critical points of $J_{g}$ in the complement of $\mathcal{D}%
_{0}^{\Gamma}$ we use the relative genus. A subset $\mathcal{Y}$ of $H_{g}%
^{1}(M)^{\Gamma}$ will be called symmetric if $-u\in\mathcal{Y}$ for every
$u\in\mathcal{Y}$.

\begin{definition}
Let $\mathcal{D}$ and $\mathcal{Y}$ be symmetric subsets of $H_{g}%
^{1}(M)^{\Gamma}$. The genus of $\mathcal{Y}$ relative to $\mathcal{D}$,
denoted by $\mathfrak{g}(\mathcal{Y},\mathcal{D})$, is the smallest number $n$
such that $\mathcal{Y}$ can be covered by $n+1$ open symmetric subsets
$\;\mathcal{U}_{0},\mathcal{U}_{1},\ldots,\mathcal{U}_{n}$ of $H_{g}%
^{1}(M)^{\Gamma}$ with the following two properties:

\begin{itemize}
\item[(i)] $\mathcal{Y}\cap\mathcal{D}\subset\mathcal{U}_{0}$ and there exists
an odd continuous map $\vartheta_{0}:\mathcal{U}_{0}\rightarrow\mathcal{D}$
such that $\vartheta_{0}(u)=u$ for $u\in\mathcal{Y}\cap\mathcal{D}$.

\item[(ii)] there exist odd continuous maps $\vartheta_{j}:\mathcal{U}%
_{j}\rightarrow\{1,-1\}$ for every $j=1,\ldots,n$.
\end{itemize}

If no such cover exists, we define $\mathfrak{g}(\mathcal{Y},\mathcal{D}%
):=\infty$.
\end{definition}

Now define
\[
c_{j}:=\inf\{c\in\mathbb{R}:\mathfrak{g}(\mathcal{D}_{c}^{\Gamma}%
,\mathcal{D}_{0}^{\Gamma})\geq j\}.
\]

\begin{lemma}
\label{Lemma:Var pri crit points}Assume that $J_{g}$ satisfies condition
$(PS)_{c_{j}}^{\Gamma}$ in $H_{g}^{1}(M).$ Then, the following statements hold true:

\begin{itemize}
\item[(a)] $J_{g}$ has a sign-changing critical point $u\in H_{g}%
^{1}(M)^{\Gamma}$ with $J_{g}(u)=c_{j}$.

\item[(b)] If $c_{j}=c_{j+1}$, then $J_{g}$ has infinitely many sign-changing
critical points $u\in H_{g}^{1}(M)^{\Gamma}$ with $J_{g}(u)=c_{j}$.
\end{itemize}

Consequently, if $J_{g}$ satisfies $(PS)_{c}^{\Gamma}$ in $H_{g}^{1}(M)$ for
every $c\leq d$, then $J_{g}$ has at least $\mathfrak{g}(\mathcal{D}%
_{d}^{\Gamma},\mathcal{D}_{0}^{\Gamma})$ pairs of sign-changing critical
points $u$ in $H_{g}^{1}(M)^{\Gamma}$ with $J_{g}(u)\leq d$.
\end{lemma}

\begin{proof}
The proof is exactly the same as that of Proposition 3.6 in \cite{cp}. It uses
the fact that $\mathcal{D}_{0}^{\Gamma}$ is strictly positively invariant
under the flow $\psi,$ and the monotonicity and subadditivity properties of
the relative genus.
\end{proof}

Now we can follow the proof of Theorem 3.7 in \cite{cp} to obtain Theorem
\ref{Theorem : Compactness}. We give the details for the sake of completeness.

\begin{proof}
[Proof of Theorem \ref{Theorem : Compactness}]Let $d:=\sup_{W}J_{g}$. By Lemma
\ref{Lemma:Var pri crit points}, we only need to show that $n:=\mathfrak{g}%
\left(  \mathcal{D}_{d}^{\Gamma},\mathcal{D}_{0}^{\Gamma}\right)  \geq
\dim(W)-1$. Let $\mathcal{U}_{0},\mathcal{U}_{1},\ldots,\mathcal{U}_{n}$ be
open symmetric subsets of $H_{g}^{1}(M)^{\Gamma}$ covering $\mathcal{D}%
_{d}^{\Gamma}$ with $\mathcal{D}_{0}^{\Gamma}\subset\mathcal{U}_{0}$ and let
$\vartheta_{0}:\mathcal{U}_{0}\rightarrow\mathcal{D}_{0}^{\Gamma}$ and
$\vartheta_{j}:\mathcal{U}_{j}\rightarrow\{1,-1\}$, $j=1,\ldots,n$, be odd
continuous maps such that $\vartheta_{0}(u)=u$ for all $u\in\mathcal{D}%
_{0}^{\Gamma}$. Since $H_{g}^{1}(M)^{\Gamma}$ is an AR we may assume that
$\vartheta_{0}$ is the restriction of an odd continuous map $\widetilde
{\vartheta}_{0}:H_{g}^{1}(M)^{\Gamma}\rightarrow H_{g}^{1}(M)^{\Gamma}$. Let
$\mathcal{B}$ be the connected component of the complement of the Nehari
manifold $\mathcal{N}_{g}^{\Gamma}$ in $H_{g}^{1}(M)^{\Gamma}$ which contains
the origin, and set $\mathcal{O}:=\{u\in W:\widetilde{\vartheta}_{0}%
(u)\in\mathcal{B}\}$. Then, $\mathcal{O}$ is a bounded open symmetric
neighborhood of $0$ in $W$.

Let $\mathcal{V}_{j}:=\mathcal{U}_{j}\cap\partial\mathcal{O}$. Then,
$\mathcal{V}_{0},\mathcal{V}_{1},\ldots,\mathcal{V}_{n}$ are symmetric and
open in $\partial\mathcal{O}$, and they cover $\partial\mathcal{O}$. Further,
by Lemma \ref{lem:variational_principle},
\[
\vartheta_{0}(\mathcal{V}_{0})\subset\mathcal{D}_{0}^{\Gamma}\cap
\mathcal{N}_{g}^{\Gamma}\subset\mathcal{N}_{g}^{\Gamma}\smallsetminus
\mathcal{E}_{g}^{\Gamma}.
\]
The set $\mathcal{N}_{g}^{\Gamma}\smallsetminus\mathcal{E}_{g}^{\Gamma}$
consists of two connected components; see, e.g., \cite{ccn}. Therefore, there
exists an odd continuous map $\eta:\mathcal{N}_{g}^{\Gamma}\smallsetminus
\mathcal{E}_{g}^{\Gamma}\rightarrow\{1,-1\}$. Let $\eta_{j}:\mathcal{V}%
_{j}\rightarrow\{1,-1\}$ be the restriction of the map $\eta\circ\vartheta
_{0}$ if $j=0,$ and the restriction of $\vartheta_{j}$ if $j=1,\ldots,n$. Take
a partition of the unity $\{\pi_{j}:\partial\mathcal{O}\rightarrow
\lbrack0,1]:j=0,1,\ldots,n\}$ subordinated to the cover $\{\mathcal{V}%
_{0},\mathcal{V}_{1},\ldots,\mathcal{V}_{n}\}$ consisting of even functions,
and let $\{e_{1},\ldots,e_{n+1}\}$ be the canonical basis of $\mathbb{R}%
^{n+1}$. Then, the map $\Psi:\partial\mathcal{O}\rightarrow\mathbb{R}^{n+1}$
given by
\[
\Psi(u):=\sum_{j=0}^{n}\eta_{j}(u)\pi_{j}(u)e_{j+1}%
\]
is odd and continuous, and satisfies $\Psi(u)\neq0$ for every $u\in
\partial\mathcal{O}$. The Borsuk-Ulam theorem allow us to conclude that
$\dim(W)\leq n+1,$ as claimed.
\end{proof}

\bigskip

\begin{acknowledgement}
We are grateful to J\'{e}r\^{o}me V\'{e}tois for suggesting the proof of Lemma
\ref{lem:vetois}. We wish also to thank the anonymous referee for his/her
careful reading and valuable comments.
\end{acknowledgement}


\begin{thebibliography}{99}                                                                                               %


\bibitem {ar}Ambrosetti, Antonio; Rabinowitz, Paul H.: Dual variational
methods in critical point theory and applications. J. Functional Analysis 14
(1973), 349--381.

\bibitem {ah}Ammann, Bernd; Humbert, Emmanuel: The second Yamabe invariant. J. Funct.
Anal. 235 (2006), no. 2, 377--412.

\bibitem {a1}Aubin, Thierry: \'{E}quations diff\'{e}rentielles non
lin\'{e}aires et probl\`{e}me de Yamabe concernant la courbure scalaire. J.
Math. Pures Appl. (9) 55 (1976), no. 3, 269--296.

\bibitem {a2}Aubin, Thierry: Some nonlinear problems in Riemannian geometry.
Springer Monographs in Mathematics. Springer-Verlag, Berlin, 1998.

\bibitem {bre}Bredon, Glen E.: Introduction to compact transformation groups.
Pure and Applied Mathematics, Vol. 46. Academic Press, New York-London, 1972.

\bibitem {b}Brendle, Simon: Blow-up phenomena for the Yamabe equation. J.
Amer. Math. Soc. 21 (2008), no. 4, 951--979.

\bibitem {bm}Brendle, Simon; Marques, Fernando C.: Blow-up phenomena for the
Yamabe equation. II. J. Differential Geom. 81 (2009), no. 2, 225--250.

\bibitem {ccn}Castro, Alfonso; Cossio, Jorge; Neuberger, John M.: A
sign-changing solution for a superlinear Dirichlet problem. Rocky Mountain J.
Math. 27 (1997), no. 4, 1041--1053.

\bibitem {cf1}Clapp, M\'{o}nica; Faya, Jorge: Multiple solutions to the
Bahri-Coron problem in some domains with nontrivial topology. Proc. Amer.
Math. Soc. 141 (2013), no. 12, 4339--4344.

\bibitem {cf2}Clapp, M\'{o}nica; Faya, Jorge: Multiple solutions to
anisotropic critical and supercritical problems in symmetric domains.
Contributions to nonlinear elliptic equations and systems, 99--120, Progr.
Nonlinear Differential Equations Appl., 86, Birkh\"{a}user/Springer, Cham, 2015.

\bibitem {cp}Clapp, M\'{o}nica; Pacella, Filomena: Multiple solutions to the
pure critical exponent problem in domains with a hole of arbitrary size. Math.
Z. 259 (2008), no. 3, 575--589.

\bibitem {cw}Clapp, M\'{o}nica; Weth, Tobias: Multiple solutions for the
Brezis-Nirenberg problem. Adv. Differential Equations 10 (2005), no. 4, 463--480.

\bibitem {d}Deimling, Klaus: Ordinary differential equations in Banach spaces.
Lecture Notes in Mathematics, Vol. 596. Springer-Verlag, Berlin-New York, 1977.

\bibitem {di}Ding, Wei Yue: On a conformally invariant elliptic equation on
$\mathbb{R}^{n}$. Comm. Math. Phys. 107 (1986), no. 2, 331--335.

\bibitem {td}tom Dieck, Tammo: Transformation groups. de Gruyter Studies in
Mathematics, 8. Walter de Gruyter \& Co., Berlin, 1987.

\bibitem {dj}Djadli, Zindine; Jourdain, Antoinette: Nodal solutions for scalar
curvature type equations with perturbation terms on compact Riemannian
manifolds. Boll. Unione Mat. Ital. Sez. B Artic. Ric. Mat. (8) 5 (2002), no.
1, 205--226.

\bibitem {dhr}Druet, Olivier; Hebey, Emmanuel; Robert, Fr\'{e}d\'{e}ric:
Blow-up theory for elliptic PDEs in Riemannian geometry. Mathematical Notes,
45. Princeton University Press, Princeton, NJ, 2004.

\bibitem {e}El Sayed, Safaa: Second eigenvalue of the Yamabe operator and
applications. Calc. Var. Partial Differential Equations 50 (2014), no. 3-4, 665--692.

\bibitem {hv}Hebey, Emmanuel; Vaugon, Michel: Le probl\`{e}me de Yamabe
\'{e}quivariant. Bull. Sci. Math. 117 (1993), no. 2, 241--286.

\bibitem {hv1}Hebey, Emmanuel; Vaugon, Michel: Sobolev spaces in the presence
of symmetries. J. Math. Pures Appl. (9) 76 (1997), no. 10, 859--881.

\bibitem {hen}Henry, Guillermo: Second Yamabe constant on Riemannian products.
J. Geom. Phys. 114 (2017), 260--275.

\bibitem {hol}Holcman, David: Solutions nodales sur les vari\'{e}t\'{e}s
riemanniennes. J. Funct. Anal. 161 (1999), no. 1, 219--245.

\bibitem {kms}Khuri, Marcus A.; Marques, Fernando C.; Schoen, Richard M.: A compactness
theorem for the Yamabe problem. J. Differential Geom. 81 (2009), no. 1, 143--196.

\bibitem {lp}Lee, John M.; Parker, Thomas H.: The Yamabe problem. Bull. Amer.
Math. Soc. (N.S.) 17 (1987), no. 1, 37--91.

\bibitem {mpv}Morabito, Filippo; Pistoia, Angela; Vaira, Giusi: Towering
Phenomena for the Yamabe Equation on Symmetric Manifolds. Potential Anal. 47
(2017), no. 1, 53--102.

\bibitem {o}Obata, Morio: The conjectures on conformal transformations of
Riemannian manifolds. J. Differential Geometry 6 (1971/72), 247--258.

\bibitem {p}Palais, Richard S.: The principle of symmetric criticality. Comm.
Math. Phys. 69 (1979), no. 1, 19--30.

\bibitem {pet}Petean, Jimmy: On nodal solutions of the Yamabe equation on
products. J. Geom. Phys. 59 (2009), no. 10, 1395--1401.

\bibitem {po}Pollack, Daniel: Nonuniqueness and high energy solutions for a
conformally invariant scalar equation. Comm. Anal. Geom. 1 (1993), no. 3-4, 347--414.

\bibitem {rv}Robert, Fr\'{e}d\'{e}ric; V\'{e}tois, J\'{e}r\^{o}me:
Sign-changing solutions to elliptic second order equations: glueing a peak to
a degenerate critical manifold. Calc. Var. Partial Differential Equations 54
(2015), no. 1, 693--716.

\bibitem {sa}Saintier, Nicolas: Blow-up theory for symmetric critical equations
involving the p-Laplacian. NoDEA Nonlinear Differential Equations Appl. 15
(2008), no. 1-2, 227--245.

\bibitem {s}Struwe, Michael: Variational methods. Applications to nonlinear
partial differential equations and Hamiltonian systems. Second edition.
Ergebnisse der Mathematik und ihrer Grenzgebiete, 34. Springer-Verlag, Berlin, 1996.

\bibitem {sch}Schoen, Richard: Conformal deformation of a Riemannian metric to
constant scalar curvature. J. Differential Geom. 20 (1984), no. 2, 479--495.

\bibitem {t}Trudinger, Neil S.: Remarks concerning the conformal deformation
of Riemannian structures on compact manifolds. Ann. Scuola Norm. Sup. Pisa (3)
22 (1968), 265--274.

\bibitem {v}V\'{e}tois, J\'{e}r\^{o}me: Multiple solutions for nonlinear
elliptic equations on compact Riemannian manifolds. Internat. J. Math. 18
(2007), no. 9, 1071--1111.

\bibitem {w}Willem, Michel: Minimax theorems. Progress in Nonlinear
Differential Equations and their Applications, 24. Birkh\"{a}user Boston,
Inc., Boston, MA, 1996.

\bibitem {y}Yamabe, Hidehiko: On a deformation of Riemannian structures on
compact manifolds. Osaka Math. J. 12 (1960), 21--37.
\end{thebibliography}
\end{document}